\documentclass{article}

\usepackage{graphicx}

\usepackage{amsmath, amssymb, amscd, amsthm, amsfonts}
\usepackage{natbib}

\usepackage{bbm}
\usepackage{mathtools}
\usepackage[dvipsnames]{xcolor}
  
\newtheorem{theorem}{Theorem}[section]
\newtheorem{corollary}[theorem]{Corollary}
\newtheorem{lemma}[theorem]{Lemma}
\newtheorem{remark}{Remark}

\newtheorem{assumption}{Assumption}
\usepackage[font=small]{caption}
\usepackage{setspace}
\usepackage{here}

\usepackage[english]{babel}

\usepackage[letterpaper,top=2cm,bottom=2cm,left=3cm,right=3cm,marginparwidth=1.75cm]{geometry}

\usepackage{amsmath}
\usepackage{graphicx}
\usepackage[colorlinks=true, allcolors=blue]{hyperref}

\usepackage[dont-mess-around]{fnpct}   
\usepackage{bigfoot}

\DeclareNewFootnote{AAffil}[arabic]
\DeclareNewFootnote{ANote}[fnsymbol]

\makeatletter
\def\thanksAAffil#1#2{
    \protected@xdef\@thanks{\@thanks
        \protect\footnotetextAAffil[#1]{#2}}%
}
\def\thanksANote#1#2{
    \protected@xdef\@thanks{\@thanks
        \protect\footnotetextANote[#1]{#2}}%
}
\makeatother

\title{On High-Dimensional Asymptotic Properties of Model Averaging Estimators}
\author{Ryo Ando\FootnotemarkAAffil{1}\thanksAAffil{1}{Department of Mathematical Informatics, The University
of Tokyo}
\and%
Fumiyasu Komaki\FootnotemarkAAffil{1}}

\date{}

\begin{document}
\maketitle

\begin{abstract}
 When multiple models are considered in regression problems, the model averaging method can be used to weigh and integrate the models.  In the present study, we examined how the goodness-of-prediction of the estimator depends on the dimensionality of explanatory variables when using a generalization of the model averaging method in a linear model. We specifically considered the case of high-dimensional explanatory variables, with multiple linear models deployed for subsets of these variables. Consequently, we derived the optimal weights that yield the best predictions. we also observe that the double-descent phenomenon occurs in the model averaging estimator. Furthermore, we obtained theoretical results by adapting methods such as the random forest to linear regression models. Finally, we conducted a practical verification through numerical experiments.
\end{abstract}

\section{Introduction}

Owing to recent technological advances, tasks involving high-dimensional data have become increasingly ubiquitous. In a high-dimensional environment, estimators may exhibit unique behaviors not observed in lower-dimensional settings. For example, the spherical concentration (\cite{hall2005geometric}) and double descent (\cite{advani2020high, belkin2018reconciling,hastie2022surprises, bartlett2020benign}) have been reported as phenomena peculiar to high-dimensional data. These behaviors necessitate in-depth research into the properties of high-dimensional data. \\
In the present study, we estimated the explanatory coefficients $\beta$ to predict future target data $y$ using the following random linear regression model:
\begin{align}
    y_i = x_i^\top \beta + \epsilon_i, \quad i = 1, 2, ...,n \label{linear model}
\end{align}
where each observation $x_i\in \mathbbm{R}^p$ and noise instance $\epsilon_i \in \mathbbm{R}$ is drawn i.i.d. from two independent distributions. Moreover, we assume that $\mathbbm{E}[x_i] = 0$, $\mathbbm{E}[\epsilon_i] = 0$, $\mathbbm{E}[x_ix_i^\top] = I$, and $\mathbbm{E}[\epsilon_i^2] = \sigma^2 > 0$ for $i=1,2,...,n$. To estimate the coefficients $\beta$ from $(x_i,y_i)$, we consider model averaging estimators obtained by constructing a new model through the integration of existing models, where each constituent model uses a min-norm least-squares estimator consisting of a distinct set of variables that may overlap with any of the other models’ sets. Specifically, we assume the preparation of $m$ candidate models, denoted as $M_1, M_2,..., M_m$, as follows:
\begin{align}
    M_k: \quad y = \sum_{i\in S_i} \beta_i x_i + \epsilon.
\end{align}
We estimate regression coefficients using the conventional min-norm least-squares method for each model, and obtain (\ref{moorepenrose}). After a list of candidate models is specified and their least-squares estimators are obtained, we predict the true value $\beta$ by using these estimators as follows:  
\begin{align}
    \label{modelaveraging2}
    \hat{y} = \sum_{i=1}^m w_i \hat\beta_{S_i}^\top x_{0,S_i},
\end{align}
where $w \in \mathbbm{R}^m$ is the weight vector with $\sum_{i=1}^mw_i = 1$, each $S_i$ is a subset of the features (e.g., $S_i = \{1,4,5,7,...,p-1,p\}$), $x_0$ is a given future datum, $x_{0,S}$ is the subvector of $x_0$ that contains features whose indices are in $S$, and $\hat\beta_{S_i, k}$ is the $k$-th entry of the min-norm least squares estimator $\hat\beta_{S}$ estimated using the variables represented by $x_{0,S}$. In other words,
    \begin{align}
    \label{moorepenrose}
        \hat\beta_S = X_S^{+} y,
    \end{align}
where $X_S$ denotes the $n \times |S|$ submatrix from which the columns comprising $S$ are taken. Here, $A^+$ denotes the Moore-Penrose inverse matrix of $A$. This predictor (\ref{modelaveraging2}) is equivalent to 
estimating $\beta$ using the model averaging estimator $\hat\beta_{w}$, where the $k$-th entry $\hat\beta_{w,k}$ is expressed as 
\begin{align}
    \label{modelaveraging1}
    \hat\beta_{w,k} = \sum_{i=1}^m w_i \mathbbm{1}_{\{k\in S_i\}}\hat\beta_{S_i, k}, 
\end{align}
for $k \in \mathbbm{N}$.

The selection of candidate models varies with respect to context. For several examples of candidate models, please refer to \cite{hansen2007least,hansen2012jackknife, ando2014model} and the references therein. 

In general, we consider the model averaging estimator $\hat\beta_{w,k}$, whose $k$th entry is expressed as
\begin{align}
    \hat\beta_{w,k} = \sum_{i=1}^m w_i \mathbbm{1}_{\{k\in S_i\}}\hat\beta_{S_i, k}^{T_i}, 
\end{align}
for $k\in \mathbbm{N}$, where $w\in \mathbbm{R}^m$ is the weight vector with $\sum_{i=1}^mw_i =1$, each $S_i$ is a subset of the features (e.g., $S_i = \{1,4,5,7,...,p-1,p\}$), each $T_i$ is a subset of the samples (e.g. $T_i=\{1,2,5,6,...,n-2,n\}$), and $\hat\beta_{S}^T$ is the min-norm least-squares estimator obtained by using $T_i$ and $S_i$, i.e.,
\begin{align}
        \hat\beta_S^T = X_S^{T+} y_T.
\end{align}
Here, $X_S^T$ denotes the $|T|\times |S|$ submatrix from which the rows comprising $T$ and columns comprising $S$ are taken, and $y_T$ denotes the subvector of $y$ containing features whose indices are in $T$. This estimator is equivalent to predicting the target value $y_0 = \beta^\top x_0$ using 
\begin{align}
    \label{modelaveraging3}
    \hat{y} = \sum_{i=1}^m w_i \hat\beta^{T_i\top }_{S_i}x_{0,S_i},
\end{align}
where $x_0$ is a given future datum and $x_{0,S}$ is the subvector of $x_0$ containing features whose indices are in $S$.
It is clear that the above estimators accommodate ordinal model-averaging estimators. 

This kind of estimators may emerge in the context of distributed learning. For example, let us consider the case where different feature sets are used at different locations to independently construct an estimator for the same target variable, and only the information from that estimator can be used. This study may be of some help as an answer to what weights should be used to aggregate these estimators. Also, if $T$'s are chosen uniformly at random, this would be like bagging, which are also analyzed in \cite{lejeune2020implicit, patil2022bagging}.

We analyzed this estimator under both the underparametrized $n>p$ and the overparametrized $n<p$ regime via random matrix theory (RMT), and derived optimal weight. For the estimator $\hat\beta$, we express our result in terms of the out-of-sample risk:  
\begin{align*}
    R_X(\hat\beta) \coloneqq \mathbbm{E}\left[( \hat{y}_0- \beta^\top x_0)^2\mid X\right] =  \mathbbm{E}\left[\|\hat\beta-\beta\|^2 \mid X\right],
\end{align*}
where $\hat{y}_0= \hat\beta^\top x_0$ and the expectation is taken as an independent random test sample $x_0$ from an equivalent distribution to that of the training data. Because we assume that $\mathbbm{E}[x_ix_i^\top] = I$ for $i=0,1,2,...,n$, the out-of-sample risk is reduced to an ordinal-squared risk, which can be regarded as a predictive risk.  We also observe that the double descent phenomenon occurs as a byproduct under certain conditions when using this estimator. Although the above assumptions are simple, this model reveals new insights into model averaging estimators. Furthermore, if, for example, the data $x_i \quad(i=1,...,n)$ follow a normal distribution $N(0,\Sigma)$, then under the perspective of predicting the target variable $y$, the linear model
\begin{align*}
    y= X\beta+\epsilon,
\end{align*} 
is equivalent to the following linear model:
\begin{align*}
    y=Z\Sigma^{\frac{1}{2}}\beta + \epsilon,
\end{align*}
where $Z = (z_1,...,z_n)^\top$ and $z_i \sim N(0,\mathrm{I}), \quad (i=1,...,n)$. Therefore, the analysis conducted in this study may be applicable if the data matrix is whitened.

\subsection{Contribution}
Our contributions can be summarized as follows:

    {\bf Precise analysis of model averaging estimators.} Using RMT, we calculated and characterized the goodness-of-prediction of the model averaging estimator in a linear model under the assumption that isotropic samples have been obtained. Our results also reveal that the double descent phenomenon occurs when using such an estimator. In addition, we derived the high-dimensional asymptotic behavior of each model when the samples and features were selected randomly.
    
    {\bf Optimal weights.} A model averaging estimator consists of a weight vector and several min-norm least-squares estimators. The weight vector can be optimized to achieve the prediction of true values. As we derived a precise theoretical curve for the predictive risk of this estimator, we obtained the optimal weight vector according to certain conditions assumed in this study.

The first contribution above can be regarded as the extension of Theorem 3.5 in \cite{lejeune2020implicit}, where precise asymptotic risk was obtained under the condition that the dimensionality of data does not exceed the number of samples. In contrast, we were able to deduce high-dimensional behavior of the model averaging estimator even when the dimensionality of data exceeded the number of samples.
In addition, this study partially extends the result of Theorem 1 in \cite{hastie2022surprises}, which investigated the double descent phenomenon of the min-norm least-squares estimators of linear regression. Whereas the study in question considered a single model estimator, we examined the integration of multiple model estimators.

\subsection{Framework of high-dimensional asymptotics}
This section presents relevant notations and assumptions in terms of high-dimensional asymptotics.

Unlike previous studies \cite{hansen2007least,hansen2012jackknife, ando2014model}, We use RMT to analyze the precise predictive risk of estimators. RMT enables us to describe the behavior of the eigenvalues of large
matrices (see, e.g. \cite{bai2010spectral}). These results are typically stated in terms of the spectral distribution $F_A(x) \coloneqq p^{-1}\sum^p_{k=1} \mathbbm{1}_{\{\lambda_k(A)\leq x\}}$, which is the cumulative distribution function of the eigenvalues $\lambda_k(A), k=1,...,p$ of a symmetric matrix $A$. The following high-dimensional asymptotic model is assumed for all theorems and corollaries.

\begin{assumption}[High-Dimensional Asymptotics]
\label{high-dimensional asymptotics}
    The following conditions hold:
    \begin{itemize}
        \item[1] Data $X_n \in \mathbb{R}^{n\times p}$ are generated with i.i.d. entries satisfying $\mathbbm{E}\left[X_{n,ij}\right] = 0$, $\mathrm{Var}\left[X_{n,ij}\right] = 1$ $\mathbbm{E}\left[|X_{n,ij}|^{12+\omega}\right] < \infty$ for an $\omega>0$.
        \item[2] The sample size $n\to \infty$ and dimension $p\to\infty$ as well, whereas the aspect ratio $p/n\to\gamma>0$.
        \item[3] The dimensions of candidate models $|S^n_i|\to\infty$, $|S^n_i\cap S^n_j|\to \infty$ as $n\to\infty$; conversely, $|S^n_i|/n \to \gamma_i >0$ and $|S^n_i\cap S^n_j|/n \to \gamma_{ij} > 0$ for any $i,j$.
        \item[4] The samples used in the candidate models $|T^n_i|\to\infty$ as $n\to\infty$; conversely, $|T^n_i|/n \to \eta_i>0$ and $|T^n_i\cap T^n_j|/n \to \eta_{ij}>0$ for any $i$.
        \item[5] The weight vector $w_n \in \mathbbm{R}^m$ satisfies $\sum_{i=1}^m w_{n,i} = 1$ and converges to a weight vector $w \in \mathbbm{R}^m$ as $n\to\infty$, $p\to\infty$ and $p/n\to \gamma$.
    \end{itemize}
\end{assumption}

Similar conditions have been assumed in \cite{dobriban2018high},  \cite{hastie2022surprises}, \cite{wu2020optimal} and \cite{richards2021asymptotics}.
 Unlike in previous studies (\cite{bai1998no}, \cite{bai2010spectral}, \cite{fujikoshi2022high},  \cite{dobriban2018high}, \cite{hastie2022surprises}), Assumption \ref{high-dimensional asymptotics} is more restrictive in that it includes the condition  $\mathbbm{E}\left[X_{n, ij}^{12+\omega}\right] < \infty$ for theoretical purposes. Specifically, this condition is necessary when applying the trace lemma to obtain the result of Lemma \ref{important_lemma}. However, many distributions, such as the Laplace and Gaussian distributions, satisfy these conditions. Note that the risk of the model averaging estimator does not depend only on the eigenvalues, and classical RMT methods cannot be applied under the present settings. Hence, we must prove certain RMT results that are specified in the present assumptions.

We also assume the following throughout this study:
\begin{assumption}[Deterministic Coefficients]
\label{random regression coefficients}
The regression coefficients $\beta_n \in \mathbbm{R}^p$ satisfy $\|\beta_n\|^2 = r^2$ for $r>0$.
\end{assumption}
To simplify the notation, we omit the subscript n by denoting $X_n$ as $X$, $S^n_i$ as $S_i$, $T^n_i$ as $T_i$, and $\beta_n$ as $\beta$.

\subsection{Related works}
\label{related works}
     {\bf Model averaging estimator for linear regression.} 
    Model averaging estimators have been the subject of extensive research. For example, \cite{akaike1978likelihood, akaike1979bayesian, hansen2007least, hansen2012jackknife, ando2014model} estimated target variables by weighting and adding the estimators in a linear model, with weights determined using various model selection criteria. \cite{akaike1978likelihood, akaike1979bayesian} considered the Akaike information criterion (AIC), \cite{hansen2007least} used the $C_p$ criterion (\cite{mallows2000some}), \cite{hansen2012jackknife} employed the Jackknife method, and \cite{ando2014model} considered using cross-validation to determine the weight vector. Each of these studies demonstrated that under the appropriate conditions, the weights derived using their respective selection methods were optimal for estimation. In particular, \cite{ando2014model} considered high-dimensional models and demonstrated their optimality under an unusual range of weighting definitions. However, all of these studies assumed the dimensionality of each candidate model to be less than the number of samples. Conversely, we examined behavior under a dimensionality exceeding the number of samples.

    {\bf Several model averaging methods: bagging and distributed learning.} In the present study, we categorized the data in terms of both the sample and feature indices. Such an averaging approach was also considered in \cite{lejeune2020implicit}, wherein samples and features were extracted randomly. However, we also considered the case of non-random extraction. \cite{patil2022bagging} obtained results for bagging and sub-bagging, where only samples were extracted randomly. Specifically, they examined the high-dimensional asymptotic behaviors of ridge and ridgeless estimators with respect to bagging and sub-bagging. In addition, \cite{dobriban2021distributed, dobriban2020wonder} studied model integration methods under non-random sample partitioning in the context of distributed learning.

     {\bf High-dimensional analysis via random matrix theory.} In the present study, we employed RMT \cite{bai2010spectral} to establish theoretical results. Previously, RMT has been applied for various statistical tasks. For example, \cite{ledoit2012nonlinear,ledoit2020analytical} used RMT to estimate covariance matrices. \cite{hastie2022surprises, dobriban2018high, dobriban2020wonder, dobriban2021distributed, patil2022bagging} considered the properties of ridge and ridgeless estimators, as well as linear model integration estimators, in an RMT framework. \cite{fujikoshi2022high} used RMT to examine properties related to the consistency of model selection, such as the AIC \cite{akaike1974new} and Bayesian Information Criterion (BIC, \cite{schwarz1978estimating}), for multivariate linear regression problems. \cite{hu2022misspecification} developed a method for estimating the signal-to-noise ratios (SNRs) of linear models using RMT. All aforementioned studies examined estimator behaviors and model selection criteria with high-dimensional data, demonstrating the powerful analytical performance of RMT under a high-dimensional setting.

    {\bf Double decent phenomenon in linear models.} The double descent phenomenon has been the subject of intense research, e.g. \cite{hastie2022surprises,  derezinski2020exact, patil2022bagging}. In machine learning and especially deep learning, data are often used to train models with large numbers of parameters. Conventional wisdom states that an excessive number of training parameters may lead to overfitting, deteriorating a model’s predictive performance. However, when the number of parameters exceeds a certain threshold, the opposite has been observed to occur, with predictive performance improving. This is known as the double descent phenomenon because the risk of forecasting exhibits an increase followed by a decrease (\cite{advani2020high, belkin2018reconciling}). In a linear model, increasing the number of parameters increases dimensionality. \cite{hastie2022surprises} demonstrated that when the number of samples falls below the dimensionality, predictive performance increases. \cite{derezinski2020exact} also demonstrated the double decent phenomenon in a different framework. \cite{patil2022bagging} examined the behavior of bagged ridge and ridgeless estimators in higher dimensions, with the ensemble estimator also exhibiting double descent.

\subsection{Outline}
 The remainder of this paper is organized as follows. Section 2 presents our primary theoretical results. A verification of these results through simulation studies is given in Section 3. In Section 4, we summarize our results and discuss future directions of research. Proofs of the theoretical results are provided in the Appendix.

\section{Main results}
The following subsections present the main results of this study. Proofs of the theorems used in this section are provided in the Appendix.

\subsection{Out-of-sample risk of model averaging estimators}
We begin by examining the predictive behavior of model-averaging estimators when each candidate model contains the true model. As described in the first section, we estimate $\beta$ using a model averaging estimator:
\begin{align}
    \hat\beta_{w,k} = \sum_{i=1}^m w_{n,i} \mathbbm{1}_{\{i\in S_i\}}\hat\beta^{T_i}_{S_i, k}, \nonumber
\end{align}
where $w\in \mathbbm{R}^m$ is the weight vector with $\sum_{i=1}^mw_i =1$, $S_i$ is a subset of features, and $T_i$ is a subset of samples.

We derive the following theorem for out-of-sample risk:

\begin{theorem}
\label{main_result1}
     Suppose that the assumption \ref{high-dimensional asymptotics} and \ref{random regression coefficients} holds and each candidate model contains a true model. Then, almost surely,
\begin{align}
    R_X(\hat\beta_w) \to w^\top H w,
\end{align}
where $H = (h_{ij})$ is a positive definite matrix defined as follows:
\begin{align}
    h_{ii} = 
    \begin{cases}
    \sigma^2 \frac{\gamma_i}{\eta_i-\gamma_i} &  \gamma_i < \eta_i \\
    \left(1-\frac{\eta_i}{\gamma_i}\right)r^2  + \frac{\eta_i}{\gamma_i-\eta_i}\sigma^2 & \gamma_i > \eta_i
    \end{cases} \\
    h_{ij} = 
    \begin{cases}
    \frac{\eta_{ij}\gamma_{ij}}{\eta_i\eta_j-\eta_{ij}\gamma_{ij}}\sigma^2  & \gamma_i<\eta_i, \gamma_j<\eta_j\\
    \frac{\eta_{ij}\gamma_{ij}}{\eta_{i}\gamma_{j} - \eta_{ij}\gamma_{ij}} \sigma^2 & \gamma_i < \eta_i, \eta_j < \gamma_j \\
    \frac{\eta_{ij}\gamma_{ij}}{\eta_j\gamma_{i} - \eta_{ij}\gamma_{ij}} \sigma^2 & \gamma_j < \eta_j, \eta_i < \gamma_i \\
    \frac{(\gamma_i -\eta_i)(\gamma_j-\eta_j)}{\gamma_{i}\gamma_{j}-\eta_{ij}\gamma_{ij}} r^2 + \frac{\eta_{ij}\gamma_{ij}}{\gamma_{i}\gamma_{j} - \eta_{ij}\gamma_{ij}} \sigma^2 &  \eta_i < \gamma_i, \eta_j< \gamma_j 
    \end{cases}\label{nondiag_main_result1}
\end{align}
for $i,j=1,2,...,m, i\neq j$.
\end{theorem}
\begin{proof}
    The proof is provided in the Appendix.
\end{proof}

From the above Theorem, we can immediately obtain the following result t by employing Lagrange multiplier method.
\begin{corollary}
\label{corollary1}
    Define signal-to-noise ratio $\mathrm{SNR} = r^2/\sigma^2$. If we can estimate $\mathrm{SNR}$ by a consistent estimator $\hat{\Gamma}$, then 
    \begin{align*}
        w_{\mathrm{opt}} = \frac{1}{\mathbbm{1}^\top \hat{H}^{-1} \mathbbm{1} }\hat{H}^{-1}\mathbbm{1} \to \mathrm{argmin}_{w} w^\top Hw,
    \end{align*}
    where $\mathbbm{1} = (1, 1, \cdots, 1)^\top$ and $\hat{H}=(\hat{h}_{ij})$ is defined as follows:
    \begin{align*}
        \hat{h}_{ii} = 
    \begin{cases}
     \frac{\gamma_i}{\eta_i-\gamma_i} & \gamma_i < \eta_i \\
    \left(1-\frac{\eta_i}{\gamma_i}\right)\hat{\Gamma} + \frac{\eta_i}{\gamma_i-\eta_i} & \eta_i < \gamma_i
    \end{cases} \\
    \Tilde{h}_{ij}^1 = 
    \begin{cases}
    \frac{\eta_{ij}\gamma_{ij}}{\eta_i\eta_j-\eta_{ij}\gamma_{ij}} & \gamma_i<\eta_i,\gamma_j < \eta_j\\
    \frac{\eta_{ij}\gamma_{ij}}{\eta_{i}\gamma_{j} - \eta_{ij}\gamma_{ij}} & \gamma_i < \eta_i, \eta_j < \gamma_j \\
    \frac{\eta_{ij}\gamma_{ij}}{\eta_{j}\gamma_{i} - \eta_{ij}\gamma_{ij}}  & \gamma_j < \eta_j, \eta_i < \gamma_i \\
    \frac{(\gamma_i -\eta_i)(\gamma_j-\eta_j)}{\gamma_{i}\gamma_{j}-\eta_{ij}\gamma_{ij}}\hat{\Gamma} + \frac{\eta_{ij}\gamma_{ij}}{\gamma_{i}\gamma_{j} - \eta_{ij}\gamma_{ij}} &  \eta_i< \gamma_i, \eta_j<\gamma_j 
    \end{cases}
    \end{align*}
    for $i,j=1,2,...,m, i\neq j$.
\end{corollary}

The problem of estimating $\mathrm{SNR}$ has been studied extensively ( refer to  \cite{jiang2007linear, dicker2014variance, jiang2016high, janson2017eigenprism} for examples). More recently, have been derived under high-dimensional asymptotics using RMT (\cite{hu2022misspecification}). For practical purposes, it is also necessary to estimate $\gamma_i$ and $\eta_i$, which can be estimated in $|S_i|/n$ and $|T_i|/n$, respectively.
In the following subsection, we consider more realistic situations wherein some models are misspecified.

\subsection{Out-of-sample risk of model averaging estimators: misspecified case}
In this section, we consider the case where some candidate models do not include the entire true model. Letting $\beta$ denote the regression coefficient of the true model and assuming $\|\beta_{S_i}\|^2 = r^2\kappa_i$ and $\|\beta_{(S_i\cap S_j)}\|^2 = r^2\kappa_{ij}$, where $\kappa_i \in (0,1), \kappa_{ij}\in(0,1)$ for $i,j=1,...,m, i\neq j$, we obtain the following result:

\begin{theorem}
\label{main_result2}
        Suppose that assumption \ref{high-dimensional asymptotics}, \ref{random regression coefficients} hold. Then, almost surely,
\begin{align}
    \mathbbm{E}\left[R_X(\hat\beta_w)\right] \to w^\top H_\mathrm{mis} w,
\end{align}
where the expectation is taken over all $X$ and $H_\mathrm{mis} = (h_{\mathrm{mis},ij})$ is a positive definite matrix defined as follows:
\begin{align*}
    h_{\mathrm{mis},ii} = 
    \begin{cases}
    \frac{\eta_i}{\eta_i-\gamma_i}r^2(1-\kappa_i)+ \frac{\gamma_i}{\eta_i-\gamma_i}\sigma^2 & \gamma_i < \eta_i \\
    \frac{\gamma_i}{\gamma_i-\eta_i}r^2(1-\kappa_i) + \left(1-\frac{\eta_i}{\gamma_i}\right) r^2\kappa_i + \frac{\eta_i}{\gamma_i-\eta_i} \sigma^2 & \eta_i < \gamma_i
    \end{cases}
\end{align*}
\begin{align}
    h_{\mathrm{mis},ij} = 
    \begin{cases}
    \frac{\eta_i\eta_j}{\eta_i\eta_j-\eta_{ij}\gamma_{ij}}r^2(1-\kappa_i-\kappa_j + \kappa_{ij}) + \frac{\eta_{ij}\gamma_{ij}}{\eta_{i}\eta_{j}-\eta_{ij}\gamma_{ij}} \sigma^2  & \gamma_i < \eta_i, \gamma_j < \eta_j\\
    \frac{\eta_i\gamma_{j}}{\eta_i\gamma_{j} - \eta_{ij}\gamma_{ij}}r^2(1-\kappa_i-\kappa_j + \kappa_{ij})+  \frac{\eta_i(\gamma_{j}-\eta_j)}{(\eta_i\gamma_j-\eta_{ij}\gamma_{ij})}r^2(\kappa_j -\kappa_{ij}) +\frac{\eta_{ij}\gamma_{ij}}{\eta_{i}\gamma_{j} - \eta_{ij}\gamma_{ij}} \sigma^2 & \gamma_i < \eta_i, \eta_j < \gamma_j \\
    \frac{\eta_j\gamma_{i}}{\eta_j\gamma_{i} - \eta_{ij}\gamma_{ij}}r^2(1-\kappa_i-\kappa_j + \kappa_{ij}) +\frac{\eta_j(\gamma_{i}-\eta_i)}{\eta_j\gamma_i-\eta_{ij}\gamma_{ij}}r^2(\kappa_i -\kappa_{ij})+ \frac{\eta_{ij}\gamma_{ij}}{\eta_i\gamma_{i} - \eta_{ij}\gamma_{ij}} \sigma^2 & \gamma_j < \eta_i, \eta_j < \gamma_i \\
    \frac{\gamma_{i}\gamma_j}{\gamma_{i}\gamma_{j} - \eta_{ij}\gamma_{ij}} r^2(1-\kappa_i-\kappa_j + \kappa_{ij}) +\frac{\gamma_{i}(\gamma_{j}-\eta_j)}{\gamma_i\gamma_j-\eta_{ij}\gamma_{ij}}r^2(\kappa_j -\kappa_{ij})  +\frac{\gamma_{j}(\gamma_{i}-\eta_i)}{\gamma_i\gamma_j-\eta_{ij}\gamma_{ij}}r^2(\kappa_i -\kappa_{ij}) \\
    \quad \quad \quad+ \frac{(\gamma_i -\eta_i)(\gamma_j-\eta_j)}{\gamma_{i}\gamma_{j}-\eta_{ij}\gamma_{ij}} r^2\kappa_{ij} + \frac{\eta_{ij}\gamma_{ij}}{\gamma_{i}\gamma_{j} - \eta_{ij}\gamma_{ij}} \sigma^2 &  \eta_i < \gamma_i, \eta_j<\gamma_j 
    \end{cases}\label{nondiag_main_result2}
\end{align}
for $i,j=1,2,...,m, i\neq j$.
    \end{theorem}
    \begin{proof}
        The proof is provided in the Appendix.
    \end{proof}
As before, this theorem can be combined with Lagrange's undetermined multiplier method to deduce the following result:

\begin{corollary}
\label{corollary2}
    Define signal-to-noise ratio $\mathrm{SNR} = r^2/\sigma^2$,$\mathrm{SNR}_i = r^2\kappa_i/\sigma^2$ and $\mathrm{SNR}_{ij} = r^2\kappa_{ij}/\sigma^2$ for $i,j=1,2,...,m, i\neq j$. If we can estimate $\mathrm{SNR} = r^2/\sigma^2,\mathrm{SNR}_i = r^2\kappa_i/\sigma^2,\mathrm{SNR}_{ij} = r^2\kappa_{ij}/\sigma^2$ by consistent estimators $\hat{\Gamma}, \hat{\Gamma}_i, \hat{\Gamma}_{ij}$, respectively, for $i,j=1,2,...,m, i\neq j$, then 
    \begin{align*}
        w_{\mathrm{opt}} = \frac{1}{\mathbbm{1}^\top \hat{H}_\mathrm{mis}^{-1} \mathbbm{1} }\hat{H}_\mathrm{mis}^{-1}\mathbbm{1} \to \mathrm{argmin}_{w} w^\top H_\mathrm{mis}w,
    \end{align*}
    where $\mathbbm{1} = (1, 1, \cdots, 1)^\top$ and $\hat{H}_\mathrm{mis}=(\hat{h}_{\mathrm{mis},ij})$ is defined as follows:
    \begin{align*}
    \hat{h}_{\mathrm{mis},ii} = 
    \begin{cases}
    \frac{\eta_i}{\eta_i-\gamma_i}(\hat\Gamma-\hat{\Gamma}_i)+ \frac{\gamma_i}{\eta_i-\gamma_i} & \gamma_i < \eta_i \\
    \frac{\gamma_i}{\gamma_i-\eta_i}(\hat\Gamma-\hat{\Gamma}_i) + \left(1-\frac{\eta_i}{\gamma_i}\right) \hat{\Gamma}_i + \frac{\eta_i}{\gamma_i-\eta_i}  & \eta_i < \gamma_i
    \end{cases}
\end{align*}
\begin{align*}
    \hat{h}_{\mathrm{mis},ij} = 
    \begin{cases}
    \frac{\eta_i\eta_j}{\eta_i\eta_j-\eta_{ij}\gamma_{ij}}(\hat{\Gamma}-\hat{\Gamma}_i-\hat{\Gamma}_j + \hat{\Gamma}_{ij}) + \frac{\eta_{ij}\gamma_{ij}}{\eta_{i}\eta_{j}-\eta_{ij}\gamma_{ij}}  & \gamma_i < \eta_i, \gamma_j < \eta_j\\
    \frac{\eta_i\gamma_{j}}{\eta_i\gamma_{j} - \eta_{ij}\gamma_{ij}}(\hat{\Gamma}-\hat{\Gamma}_i-\hat{\Gamma}_j + \hat{\Gamma}_{ij})+  \frac{\eta_i(\gamma_{j}-\eta_j)}{(\eta_i\gamma_j-\eta_{ij}\gamma_{ij}})(\hat{\Gamma}_j -\hat{\Gamma}_{ij}) +\frac{\eta_{ij}\gamma_{ij}}{\eta_{i}\gamma_{j} - \eta_{ij}\gamma_{ij}}  & \gamma_i < \eta_i, \eta_j < \gamma_j \\
    \frac{\eta_j\gamma_{i}}{\eta_j\gamma_{i} - \eta_{ij}\gamma_{ij}}(\hat{\Gamma}-\hat{\Gamma}_i-\hat{\Gamma}_j + \hat{\Gamma}_{ij}) +\frac{\eta_j(\gamma_{i}-\eta_i)}{\eta_j\gamma_i-\eta_{ij}\gamma_{ij}}(\hat{\Gamma}_i -\hat{\Gamma}_{ij})+ \frac{\eta_{ij}\gamma_{ij}}{\eta_i\gamma_{i} - \eta_{ij}\gamma_{ij}}  & \gamma_j < \eta_i, \eta_j < \gamma_i \\
    \frac{\gamma_{i}\gamma_j}{\gamma_{i}\gamma_{j} - \eta_{ij}\gamma_{ij}} (\hat{\Gamma}-\hat{\Gamma}_i-\hat{\Gamma}_j + \hat{\Gamma}_{ij}) +\frac{\gamma_{i}(\gamma_{j}-\eta_j)}{\gamma_i\gamma_j-\eta_{ij}\gamma_{ij}}(\hat{\Gamma}_j -\hat{\Gamma}_{ij}) +\frac{\gamma_{i}(\gamma_{j}-\eta_j)}{\gamma_i\gamma_j-\eta_{ij}\gamma_{ij}}(\hat{\Gamma}_i -\hat{\Gamma}_{ij}) \\
    \quad \quad \quad+ \frac{(\gamma_i -\eta_i)(\gamma_j-\eta_j)}{\gamma_{i}\gamma_{j}-\eta_{ij}\gamma_{ij}} \hat\Gamma_{ij} + \frac{\eta_{ij}\gamma_{ij}}{\gamma_{i}\gamma_{j} - \eta_{ij}\gamma_{ij}}  &  \eta_i < \gamma_i, \eta_j<\gamma_j 
    \end{cases}
\end{align*}
    for $i,j=1,2,...,m, i\neq j$.
\end{corollary}

It is easy to find consistent estimators of $\hat\Gamma_i$'s using a technique similar to that in \cite{dicker2014variance}, where the noise variance $\sigma^2$ and overall signal strength $r^2$ were estimated using
\begin{align}
\label{signalnoise}
    \hat\sigma^2 = \frac{p+n+1}{n(n+1)}\|y\|^2 - \frac{1}{n(n+1)}\|X^\top y\|^2, \quad \hat{r}^2 = -\frac{p}{n(n+1)} \|y\|^2 + \frac{1}{n(n+1)}\|X^\top y\|^2,  
\end{align}
respectively. If a model is specified and each element of $X$ follows an i.i.d. normal distribution with mean $0$ and variance $1$, then the above estimators (\ref{signalnoise}) are unbiased. Hence, we can estimate $\mathrm{SNR}$ using $\hat{r}^2/\hat\sigma^2$. These estimators can be determined using the following identities:
\begin{align}
    \mathbbm{E}\left[\frac{1}{n}\|y\|^2\right] = r^2 +\sigma^2, \quad \mathbbm{E}\left[\frac{1}{n^2}\|X^\top y\|^2\right] = \frac{p+n+1}{n}r^2 + \frac{p}{n} \sigma^2.
\end{align}
Under the same assumption as above (i.e., the model is specified and each element of $X$ follows an i.i.d. normal distribution with mean $0$ and variance $1$), we can also obtain the following result given an $n\times p_1$ submatrix $X_1$ of matrix $X = (X_1, X_2)$:
\begin{align}
    \mathbbm{E}\left[\frac{1}{n^2}\|X_1^\top y\|^2\right] &= \frac{n+1}{n}r_1^2 + \frac{p_1}{n} (r^2 + \sigma^2)\nonumber\\
    &= \frac{n+1}{n}r_1^2 + \frac{p_1}{n} \mathbbm{E}\left[\frac{1}{n}\|y\|^2\right] 
\end{align}
where $r_1^2$ is the signal strength of subvector $\beta_1$ of $\beta$ corresponding to $X_1$. Therefore, we can estimate $r_1^2/\sigma^2$ using
\begin{align}
    \frac{1}{n(n + 1)\hat\sigma^2}\|X_1^\top y\|^2 - \frac{p_1}{n(n + 1)\hat\sigma^2}\|y\|^2.
\end{align}
We can easily verify the consistency of this estimator under the assumption \ref{high-dimensional asymptotics}.

\subsection{Out-of-sample risk of model averaging estimators: linear ensemble case}
In this section, we consider the case in which candidate models are chosen uniformly and randomly. Specifically, we define the following assumptions:
\begin{assumption}
\label{assumptions random subset}
We assume the following hold:
    \begin{itemize}
        \item[1.] $\mathrm{Pr}(j\in S_i) = |S_i|/p$ for all j=1,...,p.
        \item[2.] $\mathrm{Pr}(j\in T_i) = |T_i|/n$ for all j=1,...,n.
        \item[3.] $|S_i|/p \to \alpha$ as $n,p\to\infty$ for any $i$.
        \item[4.] $ |T_i|/n \to \eta$ as $n,p\to\infty$ for any $i$.
    \end{itemize}
\end{assumption}
Furthermore, we consider the case where $w_i=1/m$ for all $i=1,...,m$. This estimator can be regarded as a linear version of the random forest. The following lemma is helpful in obtaining the asymptotic behavior of the estimator $\hat\beta_{\mathrm{ens}}$:
\begin{lemma}
    For any Lipschitz continuous function $f$,  
    \begin{align}
        \mathbbm{E}\left[f(|S_i\cap S_j|/p)\right] \to f(\alpha^2). 
    \end{align}
    \begin{align}
        \mathbbm{E}\left[f(|T_i\cap T_j|/n)\right] \to f(\eta^2). 
    \end{align}
\end{lemma}
\begin{proof}
    Because $\mathbbm{E}\left[\left.\frac{|S_i\cap S_j|}{p} \right||S_i|,|S_j|\right] = \frac{|S_i| |S_j|}{p^2} \to\alpha^2$, we only have to prove the following:
    \begin{align}
        \mathbbm{E}\left[\left.\left|\frac{|S_i\cap S_j|}{p} -\mathbbm{E}\left[\left.\frac{|S_i\cap S_j|}{p}\right||S_i|, |S_j|\right]\right|\right||S_i|, |S_j|\right]  \to 0.
    \end{align}
    This can be easily observed in
    \begin{align*}
        \mathrm{Var}\left(\left.\frac{|S_i\cap S_j|}{p}\right||S_i|, |S_j|\right)  = O\left(\frac{1}{p}\right).
    \end{align*}
    The latter part of the claim can be demonstrated in a similar manner, thereby completing the Proof. 
\end{proof}
From the above Lemma, we obtain the following result together with Theorem \ref{main_result2} and its proof.

\begin{theorem}
    \label{main_result3}
        Suppose that assumption \ref{high-dimensional asymptotics}, \ref{random regression coefficients} and \ref{assumptions random subset} hold. Then, almost surely,
\begin{align}
    \mathbbm{E}\left[R_X(\hat\beta_{\mathrm{ens}})\right] \to \frac{1}{m}D + \frac{m-1}{m}N, \label{ens_risk_main_result3}
\end{align}
where the expectation is taken over $X$, $S$'s and $T$'s; and $D$ and $N$ are defined as,
\begin{align*}
    D = 
    \begin{cases}
    \frac{\eta(1-\alpha)}{\eta-\gamma\alpha}r^2+ \sigma^2 \frac{\gamma\alpha}{\eta-\gamma\alpha} & \gamma\alpha < \eta, \\
    \frac{\gamma\alpha(1-\alpha)}{\gamma\alpha -\eta }r^2 + \left(1-\frac{\eta}{\gamma\alpha}\right) \alpha r^2  + \frac{\eta}{\gamma\alpha-\eta} \sigma^2 & \eta< \gamma\alpha,
    \end{cases} 
\end{align*}
\begin{align*}
    N = 
    \begin{cases}
    \frac{(1-\alpha)^2}{1-\gamma\alpha^2}r^2 + \frac{\gamma\alpha^2}{(1-\gamma\alpha^2)} \sigma^2  & \gamma\alpha < \eta,\\
    \frac{(\gamma-\eta)^2}{\gamma(\gamma-\eta^2)}r^2  + \frac{\eta^2}{\gamma - \eta^2} \sigma^2 &  \eta < \gamma\alpha ,
    \end{cases}
\end{align*}
for $i,j=1,2,...,m, i\neq j$.
\end{theorem}

Theorem \ref{main_result3} is a generalization of Theorem 3.5 in \cite{lejeune2020implicit}. Indeed, when $\alpha\gamma < \eta$, Theorem \ref{main_result3} coincides with the results in \cite{lejeune2020implicit}. From this theorem, we can determine the optimal feature and sample subset sizes by minimizing (\ref{ens_risk_main_result3}). When considering the minimization of \ref{ens_risk_main_result3}, some equivalent risk-achieving subset sizes may emerge, in which case it is preferable to choose the smallest possible size to minimize computational complexity. Because $N\le D$, the risk decreases as the number of models $m$ increases. This result was also obtained in \cite{patil2022bagging}. Observing $N$, we find an interesting property in that $N$ does not depend on $\eta$ when $\eta$ is larger than $\gamma\alpha$; likewise, $N$ does not depend on $\alpha$ when $\gamma\alpha$ exceeds $\eta$.

\section{Numerical experiment}
In this section, we numerically examine the high-dimensional behavior of out-of-sample risk for model averaging estimators. First, we verify the theoretical results for the behavior of the cross terms (off-diagonal component $h_{12}$ in Theorem \ref{main_result1}). For the diagonal component, this is just like the behavior of the min-norm least-squares estimator, where the risk increases until the sample-to-dimension ratio exceeds 1, and then decreases as shown in \cite{hastie2022surprises}. In this numerical experiment, the number of samples $n$ was $100$ and $\eta_1 = \eta_2 = 1$, in which case all samples are used for estimation for the sake of simplicity. Moreover, we assume $X_1$ and $X_2$ to be the features of the candidate models, and $X_1$ and $X_2$ to be $n\times p_1$ and $n\times p_2$ matrices, respectively, where $X_1$ and $X_2$ have i.i.d. $N(0,1)$ entries. Furthermore, $X_1$ and $X_2$ share $p_{12}$ column vectors, representing the matrix constructed from the set of $p_{12}$ column vectors by $X_{12}$. Here, we assume that 
\begin{align*}
    y = X_{12} \beta + \epsilon,
\end{align*}
is the true model, where $\beta = (r,...,r)^\top/\sqrt{p_{12}}$, $\epsilon \sim N(0,\mathrm{I}_n)$. 
Recall that $p_1/n =\gamma_1, p_2/n = \gamma_2$, and $p_{12}/n = \gamma_{12}$. $r^2 = \|\beta\|^2$ varies from $0.5$ to $4.5$, and $\sigma^2 =1$. The following four figures represent the asymptotic risk curves in (\ref{nondiag_main_result1}) for various scenarios. Figures \ref{fig:a} and \ref{fig:b} show variation in $\gamma_2$, whereas Figures \ref{fig:c} and \ref{fig:d} reflect variation in $\gamma_{12}$. In Figures \ref{fig:a} and \ref{fig:b}, the risk improves as the dimension of each candidate model exceeds that of the sample. Conversely, in Figures \ref{fig:c} and \ref{fig:d}, the risk increases with the overlap of each candidate model. Therefore, any increase in dimensionality of a candidate model must be carefully considered.

\begin{figure}[H]
        \begin{minipage}[b]{0.4\linewidth}
            \centering
            \includegraphics[width=\textwidth]{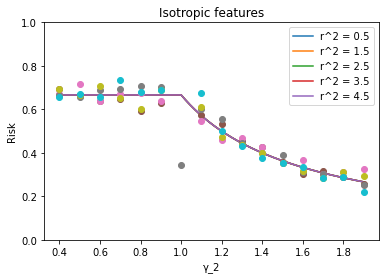}
            \caption{Asymptotic risk curves in (\ref{nondiag_main_result1}) for $h_{12}$, where $r^2$ varies from $0.5$ to $4.5$ and $\sigma^2= 1$.
            For each value of $r^2$, the points denote finite-sample risks, with $n = 100$, $p_2 = [\gamma_2 n]$ across various values of $\gamma_2$, computed from features $X_1$ and $X_2$, which are $100\times 70$ and $100\times p_2$ matrices, respectively, and have i.i.d. $N(0,1)$ entries with $p_{12} = 40$ column vectors in common.}
            \label{fig:a}
        \end{minipage}
        \hspace{0.5cm}
        \begin{minipage}[b]{0.4\linewidth}
            \centering
            \includegraphics[width=\textwidth]{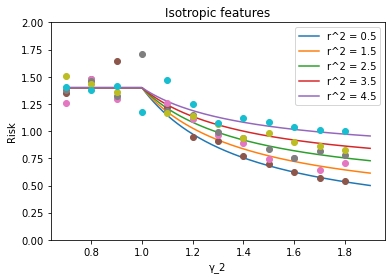}
            \caption{ Asymptotic risk curves in (\ref{nondiag_main_result1}) for $h_{12}$, where $r^2$ varies from $0.5$ to $4.5$ and $\sigma^2= 1$.
            For each value of $r^2$, the points denote finite-sample risks, with $n = 100$, $p_2 = [\gamma_2 n]$, across various values of $\gamma_2$, computed from features $X_1$ and $X_2$, which are $100\times 120$ and $100\times p_2$ matrices, respectively, and have i.i.d. $N(0,1)$ entries with $p_{12} = 70$ column vectors in common.}
            \label{fig:b}
        \end{minipage}
\end{figure}
    
\begin{figure}[H]
        \begin{minipage}[b]{0.4\linewidth}
            \centering
            \includegraphics[width=\textwidth]{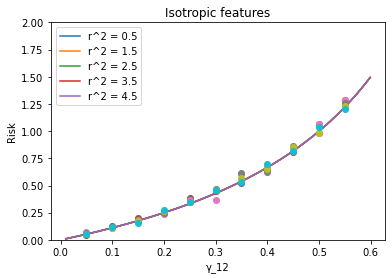}
            \caption{Asymptotic risk curves in (\ref{nondiag_main_result1}) for $h_{12}$, where $r^2$ varies from $0.5$ to $4.5$ and $\sigma^2= 1$.
            For each value of $r^2$, the points denote finite-sample risks, with $n = 100$, $p_{12} = [\gamma_{12} n]$, across various values of $\gamma_{12}$, computed from features $X_1$ and $X_2$, which are $100\times 60$ and $100\times 80$ matrices, respectively and have i.i.d. $N(0,1)$ entries with $p_{12}$ column vectors in common.}
            \label{fig:c}
        \end{minipage}
        \hspace{0.5cm}
        \begin{minipage}[b]{0.4\linewidth}
            \centering
            \includegraphics[width=\textwidth]{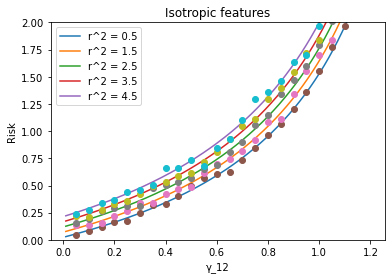}
            \caption{Asymptotic risk curves in (\ref{nondiag_main_result1}) for $h_{12}$, where $r^2$ varies from $0.5$ to $4.5$ and $\sigma^2= 1$.
            For each value of $r^2$, the points denote finite-sample risks, with $n = 100$, $p_{12} = [\gamma_{12} n]$, across various values of $\gamma_{12}$, computed from features $X_1$ and $X_2$, which are $100\times 120$, $100\times 140$ matrices, respectively, and have i.i.d. $N(0,1)$ entries with $p_{12}$ column vectors in common.}
            \label{fig:d}
        \end{minipage}
\end{figure}

Next, we conduct a numerical study of the cross terms ($h_{\mathrm{mis},12}$ of Theorem \ref{main_result2}) with a sample size $n$ of $500$. As in the previous numerical studies, we assume $X_1$ and $X_2$ to be $n\times p_1$ and $n\times p_2$ matrices, respectively, representing the features of candidate models. Both of these matrices have i.i.d. $N(0,1)$ entries and share $p_{12}$ column vectors. We denote the matrix constructed from the set of $p_{12}$ column vectors by $X_{12}$. Here, we assume that 
\begin{align*}
    y = X \beta + \epsilon,
\end{align*}
where $\beta = (0,..., 0,1, ...,1,0,...,0)^\top/\sqrt{p_0}$ (the number of $1$ is $p_0$), $\epsilon \sim N(0,\mathrm{I}_n)$ is the true model, and $X$ is an $n\times p$ matrix with i.i.d. $N(0,1)$ entries. Thus, both $X_1$ and $X_2$ are submatrices of $X$. For this numerical experiment, we assumed that $p=1000$ and the support of $\beta$ consist of $p_0$ $1$'s, with $p_{02}$ included in the part of Model 2 other than the part in common with Model 1, and $p_{01}$ included in the part Model 1 other than the part in common with Model 2. In addition, we assume that the common parts of Models 1 and 2 are included in the support. Figure \ref{fig:g} presents asymptotic risk curves in (\ref{ens_risk_main_result3}) for the cross terms ($h_{\mathrm{mis},12}$ in Theorem \ref{main_result2}) when $\sigma^2= 1$. The points denote finite-sample risks with $n = 500$ across various values of $\gamma_2$ computed from feature $X$. On the other hand, Figure \ref{fig:h} shows the case of $\gamma_{12}=1.2$ in the settings of Figure \ref{fig:g}. As observed from these figures, the risk exhibits a decrease followed by a gradual increase with the growth of Model 2. In the misspecified scenario, the double descent phenomeno also occurs.
\begin{figure}[H]
        \begin{minipage}[b]{0.4\linewidth}
            \centering
            \includegraphics[width=\textwidth]{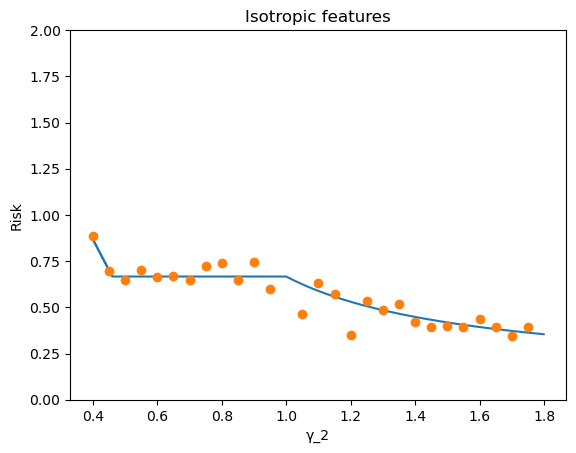}
            \caption{Asymptotic risk curves in (\ref{nondiag_main_result2}) for $h_{\mathrm{mis}, 12}$, when $p_0=250$, $p_{01} =20$, $p_{02} = 30$, $p_1=300$, $p_{12}=200$, $r^2=1$ and $\sigma^2= 1$.
            The points denote finite-sample risks averaged over $3000$ experiments, with $n = 500$, $p_{2} = [\gamma_{2} n]$, across various values of $\gamma_{2}$, computed from features $X_1$ and $X_2$ that are $500\times 300$, $500\times p_2$ matrix and have i.i.d. $N(0,1)$ entries with $200$ column vectors in common.}
            \label{fig:g}
        \end{minipage}
        \hspace{0.5cm}
        \begin{minipage}[b]{0.4\linewidth}
            \centering
            \includegraphics[width=\textwidth]{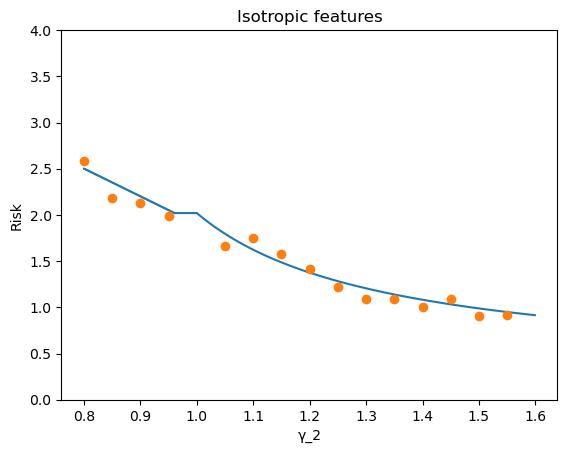}
            \caption{Asymptotic risk curves in (\ref{nondiag_main_result2}) for $h_{\mathrm{mis}, 12}$, when $p_0=500$, $p_{01} =20$, $p_{02} = 80$, $p_1=600$, $p_{12}=400$, $r^2=1$ and $\sigma^2= 1$.
            The points denote finite-sample risks averaged over $3000$ experiments, with $n = 500$, $p_{2} = [\gamma_{2} n]$, across various values of $\gamma_{2}$, computed from features $X_1$ and $X_2$ that are $500\times 600$, $500\times p_2$ matrix and have i.i.d. $N(0,1)$ entries with $400$ column vectors in common.}
            \label{fig:h}
        \end{minipage}
\end{figure}

Finally, we conducted a numerical study on the linear ensemble estimator $\hat\beta_{\mathrm{ens}}$ with a sample size $n$ of $200$. Figure \ref{fig:e} presents asymptotic risk curves in (\ref{ens_risk_main_result3}) for the linear ensemble estimator when $m$ varies from $2$ to $30$, $\eta=1$, $r^2=3.5$, and $\sigma^2= 1$. For each value of $m$, the points denote finite-sample risks, with $n = 200$, across various values of $\alpha$, computed from feature $X$, which is a $200\times 400$ matrix with i.i.d. $N(0,1)$ entries. Figure \ref{fig:f} represents the case with $\eta=0.6$ in the settings shown in Figure \ref{fig:e}. In Figure \ref{fig:e}, when $m$ is sufficiently large, the model averaging estimator obtained via random feature extraction has the same risk as the min-norm least-squares estimator when each model’s dimensionality exceeds the sample size; thus, the risk is not improved. Conversely, when the dimensionality of each model is smaller than the sample size, an improvement is observed in the risk. In addition, Figure \ref{fig:f} shows that when a sample is randomly extracted along with the features, there is no improvement, even if the dimensionality of each model is reduced.

\begin{figure}[H]
        \begin{minipage}[b]{0.4\linewidth}
            \centering
            \includegraphics[width=\textwidth]{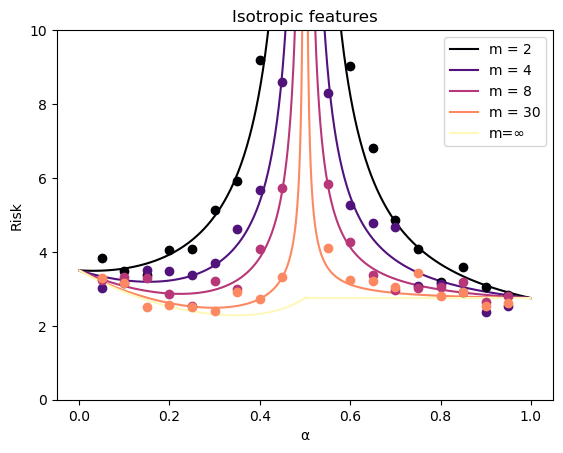}
            \caption{Asymptotic risk curves in (\ref{ens_risk_main_result3}) for linear ensemble estimator, when $m$ varies from $2$ to $30$, $\eta=1$, $r^2=3.5$ and $\sigma^2= 1$.
            For each value of $m$, the points denote finite-sample risks averaged over $300$ experiments, with $n = 200$, across various values of $\alpha$, computed from features $X$ that is $200\times 400$ matrix and have i.i.d. $N(0,1)$ entries.}
            \label{fig:e}
        \end{minipage}
        \hspace{0.5cm}
        \begin{minipage}[b]{0.4\linewidth}
            \centering
            \includegraphics[width=\textwidth]{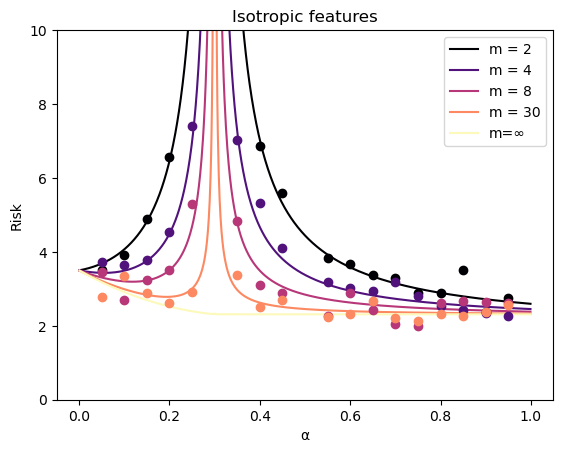}
            \caption{Asymptotic risk curves in (\ref{ens_risk_main_result3}) for linear ensemble estimator, when $m$ varies from $2$ to $30$, $\eta=0.6$, $r^2=3.5$ and $\sigma^2= 1$.
            For each value of $m$, the points denote finite-sample risks averaged over $300$ experiments, with $n = 200$, across various values of $\alpha$, computed from features $X$ that is $200\times 400$ matrix and have i.i.d. $N(0,1)$ entries.}
            \label{fig:f}
        \end{minipage}
\end{figure}

\section{Conclusion}

In this study, we derived higher-dimensional asymptotic limits for model averaging estimator, as well as a method for determining the optimal weights. As a result, we demonstrated that the double descent phenomenon occurs in model averaging estimators under certain conditions. Although many previous studies on model averaging methods achieved optimal weighting by minimizing certain criteria, we developed a method to calculate optimal weights directly by computing the asymptotic values of the predicted risk. We also note that the risk considered in this study was slightly different from those reported in previous studies (e.g., \cite{hansen2007least, ando2014model}). Whereas those studies focused on minimizing the estimated in-sample risk, we developed a framework with the objective of minimizing the out-of-sample estimated risk.

Subsequent studies will include a derivation of better estimators of the signal-to-noise ratios $\mathrm{SNR}_i$ and $\mathrm{SNR}_{ij}$ for each candidate model in Corollary \ref{corollary2}; although we derived a consistent estimator, its convergence rate were very slow numerically. In addition, it is important to consider the case where $\Sigma\neq I$. To this end, we believe that the results of the general sample covariance matrix deduced in \cite{yin2018no,yin2022some} can be extended further.

\section{Acknowledgement}

This work was supported by JSPS KAKENHI Grant Number 22H00510,
and AMED Grant Numbers JP23dm0207001 and JP23dm0307009.

\bibliographystyle{apalike}
\bibliography{reference}

\renewcommand{\thesection}{\Alph{section}}

\setcounter{section}{0}
\section{Appendix}

We now prove Theorem \ref{main_result1}.

\subsection{Proof of Theorem \ref{main_result1}}
\label{proof_of_main_result1}
\subsubsection{Proof sketch for Theorem \ref{main_result1}}
We first briefly describe the proof of Theorem \ref{main_result1}. 
\begin{proof}[proof sketch for Theorem \ref{main_result1}]
    Let $|T_i| = n_i$, $|S_i|=p_i$, $|T_{i}\cap T_{j}| = n_{ij}$ and $|S_i\cap S_j| = p_{ij}$. First, we consider the ridge regression estimator $\hat\beta_{w_n,\lambda}$ as follows:
        \begin{align*}
            \hat\beta_{w_n,\lambda,k} = \sum_{i=1}^m w_{n,i} \mathbbm{1}_{\{i\in S_i\}}\hat\beta^{T_i}_{S_i,\lambda, k}, 
        \end{align*}
        where $\hat\beta_{w,\lambda,k}$ is the $k$-th component of $\hat\beta_{w_n,\lambda}$ and $\hat\beta^{T_i}_{S_i,\lambda} = \frac{1}{n_i}(\frac{1}{n_i}X^{T_i\top}_{S_i} X^{T_i}_{S_i} +\lambda \mathrm{I})^{-1}X^{T_i\top}_{S_i} y_{T_i}$, for $i=1,...,m$,  $\lambda>0$.
        Define 
        \begin{align*}
            B_D(\hat\beta_{w_n,\lambda}) &= \sum_{i=1}^m w_{n,i}^2 \lambda^2\beta_{S_i}^\top\left(\frac{1}{n_i}X^{T_i\top}_{S_i} X^{T_i}_{S_i}+\lambda I\right)^{-2}\beta_{S_i},\\  V_D(\hat\beta_{w_n,\lambda}) &= \sum_{i=1}^m w_{n,i}^2\frac{\sigma^2}{n_i^2} \mathrm{Tr}\left(X^{T_i}_{S_i}\left(\frac{1}{n_i}X^{T_i\top}_{S_i} X^{T_i}_{S_i}+ \lambda \mathrm{I}\right)^{-2}X^{T_i\top}_{S_i}\right)\\
            B_N(\hat\beta_{w_n,\lambda}) &=  \sum_{i,j, i\neq j} w_{n,i}w_{n,j} \lambda^2\beta_{S_i}^\top\left(\frac{1}{n_i}X^{T_i\top}_{S_i}  X^{T_i}_{S_i}+\lambda \mathrm{I}\right)^{-1}E_{i,j}\left(\frac{1}{n_j}X^{T_j\top}_{S_j}  X^{T_j}_{S_j}+\lambda \mathrm{I}\right)^{-1} \beta_{S_j},\\
            V_N(\hat\beta_{w,\lambda}) &= \sum_{i,j, i\neq j} w_{n,i}w_{n,j}\frac{\sigma^2}{n_in_j}  \mathrm{Tr}\left(X^{T_i\cap T_j}_{S_i}\left(\frac{1}{n}X^{T_i\top}_{S_i}  X^{T_i}_{S_i} +\lambda \mathrm{I}\right)^{-1}E_{i,j}\left(\frac{1}{n}X^{T_j\top}_{S_j}  X^{T_j}_{S_j}+\lambda \mathrm{I}\right)^{-1}X^{T_i\cap T_j \top}_{S_j} \right),
        \end{align*}
    where $E_{i,j} = \mathbb{E}\left[x_{0,S_i}x_{0,S_j}^\top\right]$, $B_D(\hat\beta_{w_n,\lambda})$ and $V_D(\hat\beta_{w_n,\lambda})$ correspond to ``diagonal component", representing a kind of covariance between the same linear models, and $B_N(\hat\beta_{w_n,\lambda})$ and $V_N(\hat\beta_{w_n,\lambda})$ correspond to ``off-diagonal component", representing a kind of covariance between different linear models.
    Then    
    \begin{align}
        R_X(\hat\beta_{w_n,\lambda}) = B_D(\hat\beta_{w_n,\lambda}) + V_D(\hat\beta_{w_n,\lambda}) + B_N(\hat\beta_{w_n,\lambda}) +V_N(\hat\beta_{w_n,\lambda}).
    \end{align}
    Observing the above terms, it is apparent that Lemma \ref{fundamental lemma} can be applied. Subsequently, Lemma \ref{important_lemma} can be used to complete the proof.
\end{proof}

\subsubsection{Detailed proof of Theorem \ref{main_result1}}

\begin{proof}[Proof of Theorem \ref{main_result1}]
    
     In the following, we focus on $B_D(\hat\beta_{w_n,\lambda}), V_D(\hat\beta_{w_n,\lambda}), B_N(\hat\beta_{w_n,\lambda})$ and $V_N(\hat\beta_{w_n,\lambda})$. For the notation convenience, let
    \begin{align}
        e_l(-\lambda)  = \frac{-1+\frac{p_l}{n_l} -\lambda + \sqrt{(1-\frac{p_l}{n_l}+\lambda)^2+4p_l \lambda/n_l}}{2p_l \lambda/n_l},
    \end{align}
    for $l = 1,2,...,m$. Thus, when $\frac{p_l}{n_l} < 1$, 
    \begin{align}
        \lim_{\lambda\to 0} e_l(-\lambda) = \frac{1}{1-\frac{p_l}{n_l}},
    \end{align}
    on the other hand, when $\frac{p_l}{n_l} > 1$,
    \begin{align}
        \lim_{\lambda\to 0} \lambda e_l(-\lambda) = 1-\frac{n_l}{p_l}.
    \end{align}
    
    Note that from Theorem 1 and its corollary of \cite{bai1998no} (see also Chapter 6 of \cite{bai2010spectral}), the largest and smallest nonzero eigenvalues of $X^{T_i\top}_{S_i}  X^{T_i}_{S_i}$ have upper and lower constant bounds $M$ and $1/M$ independent of $i$ and $p$, respectively, for all large $n$. Hence, because $\sum_{i=1}^n w_i = 1$, $B_D(\hat\beta_{w_n,\lambda}), V_D(\hat\beta_{w_n,\lambda}), B_N(\hat\beta_{w_n,\lambda})$ and $V_N(\hat\beta_{w_n,\lambda})$ are all upper-bounded by a constant for all large $n$ and $\lambda>0$. Furthermore, we can also bound the derivative of $B_D(\hat\beta_{w_n,\lambda}), V_D(\hat\beta_{w_n,\lambda}), B_N(\hat\beta_{w_n,\lambda})$ and $V_N(\hat\beta_{w_n,\lambda})$ from above.

    We first consider the case wherein $\gamma_i < \eta_i, \gamma_j < \eta_j$ for any $i,j$. From Theorem 1 and its Corollary in \cite{bai1998no} (see also chapter 6 of \cite{bai2010spectral}),  all eigenvalues of $X^{T_i\top}_{S_i}  X^{T_i}_{S_i}$ are bounded away from 0 for all large $n$ almost surely, $B_D(\hat\beta_{w_n,\lambda})\to 0$, $B_N(\hat\beta_{w_n,\lambda}) \to 0$ as $\lambda\to 0$. Therefore, we only need to focus on the variance terms. From (9) in Theorem 1 of 
    \cite{hastie2022surprises}, $V_D(\hat\beta_{w_n,\lambda}) \to \sum_{i=1}^m w_{i}^2 \frac{\gamma_i}{1-\gamma_i}$ as $n\to \infty, \lambda \to 0$. For the term $V_N(\hat\beta_{w_n,\lambda})$, Lemmas \ref{important_lemma} and \ref{fundamental lemma} help us obtain its asymptotic behavior. If we replace $Z_1$ with $X_{S_i}^{T_i}$ and $Z_2$ with $X^{T_j}_{S_j}$ in Lemma \ref{fundamental lemma}, we obtain
    \begin{align}
    V_N(\hat\beta_{w_n,\lambda}) &= \sum_{i,j, i\neq j} w_{n,i}w_{n,j}\sigma^2  \frac{1}{n_in_j}\mathrm{Tr}\left\{X_{S_i}^{T_i\cap T_j}\left(\frac{1}{n_i}X_{S_i}^{T_i\top}  X_{S_i}^{T_i} + \lambda \mathrm{I}\right)^{-1}E_{i,j}\left(\frac{1}{n_j}X_{S_j}^{T_j\top}  X_{S_j}^{T_j}+\lambda \mathrm{I}\right)^{-1}X_{S_j}^{T_i\cap T_j \top} \right\}\nonumber\\
    &=\frac{1}{n_{ij}}\sum_{i,j, i\neq j} w_{n,i}w_{n,j}\sigma^2\frac{\frac{n_{ij}}{n_in_j}\mathrm{Tr} \left(\left(\frac{1}{n_i}X_{S_i}^{T_i\top}  X_{S_i}^{T_i} +\lambda \mathrm{I}\right)^{-1}E_{i,j}\left(\frac{1}{n_j}X_{S_j}^{T_j\top}  X_{S_j}^{T_j}+\lambda \mathrm{I}\right)^{-1} E_{i,j}^\top\right)}{\left(1+\frac{1}{n_i}\mathrm{Tr}\left(\left(\frac{1}{n}X_{S_i}^{T_i\top}  X_{S_i}^{T_i}+\lambda \mathrm{I}\right)^{-1}\right)\right)\left(1+\frac{1}{n_j}\mathrm{Tr}\left(\left(\frac{1}{n_j}X_{S_j}^{T_j\top}  X_{S_j}^{T_j}+\lambda \mathrm{I}\right)^{-1}\right)\right)} \nonumber\\ 
    & \quad + o_{a.s.}(1). \label{eqn:4.1}
\end{align}
Hence, by applying Lemma \ref{important_lemma} to (\ref{eqn:4.1}), we obtain 
\begin{align}
    V_N(\hat\beta_{w_n,\lambda}) = \sum_{i,j, i\neq j} w_{n,i}w_{n,j}\sigma^2 \frac{\frac{n_{ij}p_{ij}}{n_in_j}e_i(-\lambda)e_j(-\lambda)}{(1+\frac{p_i}{n_i}e_i(-\lambda))(1+\frac{p_j}{n_j}e_j(-\lambda))\left(1-\frac{\frac{n_{ij}p_{ij}}{n_in_j}e_{i}(-\lambda)e_{j}(-\lambda)}{(1+\frac{p_i}{n_i}e_i(-\lambda))(1+\frac{p_j}{n_j}e_j(-\lambda))}\right)} 
     \quad + o_{a.s.}(1). 
\end{align}
As mentioned previously, $B_D(\hat\beta_{w_n,\lambda}), V_D(\hat\beta_{w_n,\lambda}), B_N(\hat\beta_{w_n,\lambda})$ and $V_N(\hat\beta_{w_n,\lambda})$ along with their derivatives, are all bounded away from $\infty$ and are equicontinuous and uniformly bounded in terms of $\lambda \ge 0$. Therefore we can apply Arzel\`a-Ascoli theorem to obtain the claim by $\lambda\to0$ after $n,p\to \infty$. 

Next, we consider the case where $\gamma_i < \eta_i, \eta_j < \gamma_j$. Here, the asymptotic behavior of the variance term can be obtained in a manner similar to the case $\gamma_i < \eta_i, \gamma_j < \eta_j$. From Theorem 1 and its Corollary in \cite{bai1998no} (see also chapter 6 of \cite{bai2010spectral}), the eigenvalues of $X^{T_i\top}_{S_i}  X^{T_i}_{S_i}$ are bounded away from 0; Thus, $B_D\to0$ and $B_N\to0$ as $n,p\to\infty$ and $\lambda\to 0$. The cases in which $\gamma_j < \eta_j, \eta_i < \gamma_i$ are similar.

Finally, we consider the case where $\eta_i<\gamma_i, \eta_j<\gamma_j$. In this case, the same argument as in the previous two cases yields the asymptotic behavior of the variance term. As for the bias term $B_D(\hat\beta_{w_n,\lambda})$, we can obtain the claim using an argument similar to the proof of Theorem 1 in \cite{hastie2022surprises}. In terms of $B_N(\hat\beta_{w_n,\lambda})$, by substituting $E_{12}$ into $\Theta_1$ and $\beta_{S_j}\beta_{S_i}^\top$ into $\Theta_2$ in Lemma \ref{important_lemma}, we obtain
\begin{align}
    B_N(\hat{\beta}_{w_n,\lambda}) &=\sum_{i,j, i\neq j} w_{n,i}w_{n,j}\lambda^2\beta_{S_i}^\top\left(X^{T_i\top}_{S_i}  X^{T_i}_{S_i}+\lambda \mathrm{I}\right)^{-1}E_{i,j}\left(X^{T_j\top}_{S_j}  X^{T_j}_{S_j}+\lambda \mathrm{I}\right)^{-1} \beta_{S_j} \nonumber\\
    &= \sum_{i,j, i\neq j} w_{n,i}w_{n,j} \left(\lambda^2 e_ie_j\beta_{S_i}^\top E_{i,j} \beta_{S_j} + \lambda^2e_ie_j\frac{\beta_{S_i}^\top E_{i,j} \beta_{S_j}\frac{n_{ij}p_{ij}}{n_in_j}e_i(-\lambda)e_j(-\lambda)}{(1+\frac{p_i}{n_i}e_i)(1+\frac{p_j}{n_j}e_j)\left(1-\frac{\frac{n_{ij}p_{ij}}{n_in_j}e_{i}(-\lambda)e_{j}(-\lambda)}{(1+\frac{p_i}{n_i}e_i(-\lambda))(1+\frac{p_j}{n_j}e_j(-\lambda))}\right)}\right) + o_{a.s.}(1).
\end{align}
Therefore, by applying Arzel\`a-Ascoli theorem and $\lambda\to 0$ after $n,p\to\infty$, and performing some calculations, we obtain
\begin{align}
    B_N(\hat{\beta}_{w_n}) \to  \sum_{i,j, i\neq j} w_iw_j\frac{(\gamma_i -\eta_j)(\gamma_j-\eta_i)}{\gamma_{i}\gamma_{j}-\eta_{ij}\gamma_{ij}} r^2\kappa_{ij}.
\end{align}

This completes the proof.

\end{proof}

In the next, we prove Theorem \ref{main_result2}.

\subsection{proof of Theorem \ref{main_result2}}
\label{proof_of_main_result2}

\subsubsection{Proof sketch for Theorem \ref{main_result2}}
\begin{proof}[proof sketch for Theorem \ref{main_result2}]
    Because analogous proofs can be used for the diagonal and off-diagonal components of $H_{\mathrm{mis}}$, we hereafter focus specifically on the off-diagonal part.
    The off-diagonal part can be written as
    \begin{align*}
        &\mathbbm{E}\left[(\hat\beta^{T_i\top}_{S_i}  x_{0, S_i} - \beta^\top x_0)(\hat\beta^{T_j\top}_{S_j}  x_{0, S_j} - \beta^\top x_0)\right] \\
        &=\mathbbm{E}\left[(\hat\beta^{T_i\top}_{S_i}  x_{0, S_i} - \beta_{S_i}^\top x_{0, S_i})(\hat\beta^{T_j\top}_{S_j}  x_{0, S_j} - \beta_{S_j}^\top x_{0,S_j})\right] + \mathbbm{E}\left[(\beta_{S_i^c \cap S_j^c}^\top x_{0,S_i \cup S_j})^2\right]\\
        & \quad - \mathbbm{E}\left[(\hat\beta^{T_i\top}_{S_i}  x_{0, S_i} - \beta_{S_i}^\top x_{0,S_i}) \beta_{S_i \setminus S_j}^\top x_{0,S_i \setminus S_j}\right]
        - \mathbbm{E}\left[(\hat\beta^{T_j\top}_{S_j}  x_{0, S_j} - \beta_{S_j}^\top x_{0,S_j}) \beta_{S_j \setminus S_i}^\top x_{0, S_j \setminus S_i}\right]\\
        &= \Delta_1 + \Delta_2 + \Delta_3 + \Delta_4 + \Delta_5 + \Delta_6 + \Delta_7
    \end{align*}
    where 
    \begin{align}
        \Delta_1 = \mathbbm{E}\left[(\Tilde\beta^{T_i}_{S_i,S_j}- \beta_{S_i})^\top E_{i,j}(\Tilde\beta^{T_j}_{S_j,S_i} - \beta_{S_j})\right],
    \end{align}
    \begin{align}
        \Delta_2 = \mathbbm{E}\left[\beta_{S_j\setminus S_i}^\top X^{T_i\top}_{S_j\setminus S_i}  X^{T_i}_{S_i} (X^{T_i\top}_{S_i}  X^{T_i}_{S_i})^+ E_{i,j}(\Tilde\beta^{T_j}_{S_j, S_i} - \beta_{S_j})\right],
    \end{align}
    \begin{align}
        \Delta_3 = \mathbbm{E}\left[(\Tilde\beta^{T_i}_{S_i,  S_j}- \beta_{S_i})^\top E_{i,j}(X^{T_j\top}_{S_j}  X^{T_j}_{S_j})^+ X^{T_j\top}_{S_j}  X^{T_j}_{S_i\setminus S_j}\beta_{S_i\setminus S_j}\right],
    \end{align}
    \begin{align}
        \Delta_4 = \mathbbm{E}\left[\beta_{S_j\setminus S_i}^\top X^{T_i\top}_{S_j\setminus S_i}  X^{T_i}_{S_i} (X^{T_i\top}_{S_i}  X^{T_i}_{S_i})^+E_{i,j}(X^{T_j\top}_{S_j}  X^{T_j}_{S_j})^+ X^{T_j\top}_{S_j}  X^{T_j}_{S_i\setminus S_j}\beta_{S_i\setminus S_j}\right],
    \end{align}
    \begin{align}
        \Delta_5 = \mathbbm{E}\left[\left(\beta_{S_i^c \cap S_j^c}x_{0,S_i^c \cap S_j^c}\right)^2\right] = \|\beta_{S_i^c \cap S_j^c}\|^2,
    \end{align}
    \begin{align}
        \Delta_6 = \mathbbm{E}\left[(\beta_{S_j} - \hat\beta^{T_j}_{S_j})^\top F_{j,i} \beta_{S_j \setminus S_i}\right],
    \end{align}
    \begin{align}
         \Delta_7 = \mathbbm{E}\left[(\beta_{S_i} - \hat\beta^{T_i}_{S_i})^\top F_{i,j}\beta_{S_i \setminus S_j}\right],
    \end{align}
    where $\Tilde\beta^{T_l}_{S_l, S_m} = (X^{T_l\top}_{S_l}  X^{T_l}_{S_l})^+ X^{T_l\top}_{S_l}  (X^{T_l}_{S_l} \beta_{S_l} + X^{T_l}_{S_l^c\cap S_m^c} \beta_{S_l^c\cap S_m^c}+\epsilon_{T_l})$, $F_{l,m} = \mathbbm{E}\left[x_{0,S_l} x_{S_l\setminus S_m}^\top\right]$. For $\Delta_1$, because we can regard the term  $X^{T_l}_{S_l^c\cap S_m^c} \beta_{S_l^c\cap S_m^c}+\epsilon_{T_l}$ of $\Tilde\beta^{T_l}_{S_l, S_m}$ as noise term, we can apply the same procedure as in the proof for Theorem \ref{main_result1}. We need to obtain the limiting behavior of the rest of the terms $\Delta_2, \Delta_3, \Delta_4, \Delta_6$ and $\Delta_7$. Because the derivation is somewhat involved, please refer to the following section for detail. 
\end{proof}

\subsubsection{Detailed proof of Theorem \ref{main_result2}}

\begin{proof}[proof of Theorem \ref{main_result2}]
     
     For $\Delta_1$, because we can regard the term  $X^{T_l}_{S_l^c\cap S_m^c} \beta_{S_l^c\cap S_m^c}+\epsilon_{T_l}$ of $\Tilde\beta^{T_l}_{S_l, S_m}$ as the noise term, using the same procedure as in the proof for Theorem \ref{main_result1}, we can obtain, as $n,p\to\infty$,
    \begin{align*}
        \Delta_1 \to 
        \begin{cases}
            \frac{\eta_{ij}\gamma_{ij}}{\eta_i\eta_j-\eta_{ij}\gamma_{ij}}(r^2(1-\kappa_i-\kappa_j + \kappa_{ij}) + \sigma^2) \quad &\gamma_i< \eta_i, \gamma_j < \eta_j,\\
            \frac{\eta_{ij}\gamma_{ij}}{\eta_{i}\gamma_j -\eta_{ij}\gamma_{ij}} (r^2(1-\kappa_i-\kappa_j + \kappa_{ij}) + \sigma^2) \quad & \gamma_i<\eta_i, \eta_j<\gamma_j,\\
            \frac{\eta_{ij}\gamma_{ij}}{\eta_j\gamma_i -\eta_{ij}\gamma_{ij}} (r^2(1-\kappa_i-\kappa_j + \kappa_{ij})+ \sigma^2) \quad & \gamma_j<\eta_j, \eta_i<\gamma_i,\\
            \frac{(\gamma_i-\eta_i)(\gamma_j-\eta_j)}{\gamma_i\gamma_j - \eta_{ij}\gamma_{ij}}r^2\kappa_{ij} + \frac{\eta_{ij}\gamma_{ij}}{\gamma_i\gamma_j-\eta_{ij}\gamma_{ij}}(r^2(1-\kappa_i-\kappa_j + \kappa_{ij}) + \sigma^2) \quad &1<\gamma_i,\gamma_j,
        \end{cases}
    \end{align*}
    along with bounded convergence theorem.
    On the terms $\Delta_2,\Delta_3$, since 
    \begin{align*}
        \Delta_2 &=  \mathbbm{E}\left[\beta_{S_j\setminus S_i}^\top X^{T_i\top}_{S_j\setminus S_i}  X^{T_i}_{S_i} (X^{T_i\top}_{S_i}  X^{T_i}_{S_i})^+ E_{i,j}(\Tilde\beta^{T_j}_{S_j, S_i} - \beta_{S_j})\right]\\
        &=\mathbbm{E}\left[\beta_{S_j\setminus S_i}^\top X^{T_i\top}_{S_j\setminus S_i}  X^{T_i}_{S_i} (X^{T_i\top}_{S_i}  X^{T_i}_{S_i})^+ E_{i,j}\left(\left(X^{T_j\top}_{S_{j}}  X^{T_j}_{S_j}\right)^+ X^{T_j\top}_{S_j}  X^{T_j}_{S_j} -I\right) \beta_{S_j}\right]\\
        &=\sum_{k\in T_i}\mathbbm{E}\left[ x_{S_i,k} ^\top(X^{T_i\top}_{S_i}  X^{T_i}_{S_i})^+ E_{i,j}\left(\left(X^{T_j\top}_{S_{j}}  X^{T_j}_{S_j}\right)^+ X^{T_j\top}_{S_j}  X^{T_j}_{S_j} -I\right) \beta_{S_j}\beta_{S_j\setminus S_i}^\top x_{S_j\setminus S_i,k}\right],
    \end{align*}
    \begin{align*}
        \Delta_3 &= \mathbbm{E}\left[(\Tilde\beta^{T_i}_{S_i,  S_j}- \beta_{S_i})^\top E_{i,j}(X^{T_j\top}_{S_j}  X^{T_j}_{S_j})^+ X^{T_j\top}_{S_j}  X^{T_j}_{S_i\setminus S_j}\beta_{S_i\setminus S_j}\right]\\
        &=\mathbbm{E}\left[\beta_{S_i}^\top\left(X^{T_i\top}_{S_i}  X^{T_i}_{S_i}(X^{T_i\top}_{S_i}  X^{T_i}_{S_i})^+ -\mathrm{I}\right) E_{i,j}(X^{T_j\top}_{S_j}  X^{T_j}_{S_j})^+ X^{T_j\top}_{S_j}  X^{T_j}_{S_i\setminus S_j}\beta_{S_i\setminus S_j}\right]\\
        &= \sum_{k\in T_j}\mathbbm{E}\left[x_{S_i\setminus S_j,k}^\top \beta_{S_i\setminus S_j}\beta_{S_i}^\top\left(X^{T_i\top}_{S_i}  X^{T_i}_{S_i}(X^{T_i\top}_{S_i}  X^{T_i}_{S_i})^+ -\mathrm{I}\right) E_{i,j}(X^{T_j\top}_{S_j}  X^{T_j}_{S_j})^+ x_{S_j,k} \right]
    \end{align*}
    where $x_k$ denotes the $k$-th row vector of $X$. Let us consider
    \begin{align}
        \delta_{2,k}(\lambda) = \delta_{2,k} = -\lambda x_{S_i,k}^\top Q_i E_{i,j}Q_j \beta_{S_j}\beta_{S_j\setminus S_i}^\top x_{S_j\setminus S_i,k}, 
    \end{align}
    for $\lambda\ge 0$, where $Q_i(\lambda) = \left(\frac{1}{n_i}X^{T_i\top}_{S_i}  X^{T_i}_{S_i}+ \lambda\mathrm{I}\right)^{-1}$ and $Q_j(\lambda) = \left(\frac{1}{n_j}X^{T_j\top}_{S_j}  X^{T_j}_{S_j}+ \lambda\mathrm{I}\right)^{-1}$. Furthermore, we define $Q_{i,(k)}(\lambda) = \left(\frac{1}{n_i}X^{T_i\top}_{S_i,(k)}  X^{T_i}_{S_i,(k)}+ \lambda\mathrm{I}\right)^{-1}$ and $Q_{j,(k)}(\lambda) = \left(\frac{1}{n_j}X^{T_j\top}_{S_j,(k)}  X^{T_j}_{S_j,(k)}+ \lambda\mathrm{I}\right)^{-1}$, where $X_{(k)}$ denotes the $(n-1)\times p$ matrix excluding the $k$-th row vector of an $n\times p$ matrix $X$. Note that, 
    \begin{align}
        \mathbbm{E}\left[\lim_{\lambda \to 0}\frac{1}{n_i}\sum_{k\in T_i}\delta_{2,k}\right] =\mathbbm{E}\left[\lim_{\lambda \to 0}\frac{1}{n_i}\sum_{k\in T_i}-\lambda x_{S_i,k}^\top Q_i E_{i,j}Q_j \beta_{S_j}\beta_{S_j\setminus S_i}^\top x_{S_j\setminus S_i,k}\right] = \Delta_2.
    \end{align}
    and $\frac{1}{n_i}\sum_{k\in T_i}\delta_{2,k}$ can be uniformly bounded from above on $\lambda \in \left[0, C\right]$ for $n$ large enough and a positive real number $C>0$. This can be seen from the results of Theorem 1 and its corollary in \cite{bai1998no}.
    If $k \in T_1\cap T_2$, we can obtain
    \begin{align}
        \delta_{2,k} &= -\lambda x_{S_i,k}^\top Q_i E_{i,j}Q_j \beta_{S_j}\beta_{S_j\setminus S_i}^\top x_{S_j\setminus S_i,k} \nonumber\\
        &=  -\lambda \frac{x_{S_i,k}^\top Q_{i,(k)} E_{i,j}Q_{j,(k)} \beta_{S_j}\beta_{S_j\setminus S_i}^\top x_{S_j\setminus S_i,k}}{\left(1+\frac{1}{n_i}x_{S_i,k}^\top Q_{i,(k)}x_{S_i,k}\right)} + \lambda \frac{\frac{1}{n_j}x_{S_i,k}^\top Q_{i,(k)} E_{i,j}Q_{j,(k)}x_{S_j,k}x_{S_j,k}^\top Q_{j,(k)}\beta_{S_j}\beta_{S_j\setminus S_i}^\top x_{S_j\setminus S_i,k}}{\left(1+\frac{1}{n_i}x_{S_i,k}^\top Q_{i,(k)}x_{S_i,k}\right)\left(1 +  \frac{1}{n_j}x_{S_j,k}^\top Q_{j,(k)}x_{S_j,k}\right)}.
    \end{align}
    Furthermore, 
    \begin{align}
        \mathbbm{E}\left[\left| \lambda \frac{x_{S_i,k}^\top Q_{i,(k)} E_{i,j}Q_{j,(k)} \beta_{S_j}\beta_{S_j\setminus S_i}^\top x_{S_j\setminus S_i,k}}{\left(1+\frac{1}{n_i}x_{S_i,k}^\top Q_{i,(k)}x_{S_i,k}\right)} - \lambda \frac{x_{S_i,k}^\top Q_{i,(k)} E_{i,j}Q_{j,(k)} \beta_{S_j}\beta_{S_j\setminus S_i}^\top x_{S_j\setminus S_i,k}}{\left(1+\frac{1}{n_i}\mathrm{Tr}\left( Q_{i,(k)}\right)\right)} \right|^q\right] = O\left(\frac{1}{n^{q/2}}\right),
    \end{align}
    \begin{align}
        &\mathbbm{E}\left[\left| \lambda \frac{\frac{1}{n_j}x_{S_i,k}^\top Q_{i,(k)} E_{i,j}Q_{j,(k)}x_{S_j,k}x_{S_j,k}^\top Q_{j,(k)}\beta_{S_j}\beta_{S_j\setminus S_i}^\top x_{S_j\setminus S_i,k}}{\left(1+\frac{1}{n_i}x_{S_i,k}^\top Q_{i,(k)}x_{S_i,k}\right)\left(1 +  \frac{1}{n_j}x_{S_j,k}^\top Q_{j,(k)}x_{S_j,k}\right)} \right.\right.\\
        & \quad - \left.\left. \lambda \frac{\frac{1}{n_j} \mathrm{Tr}\left(Q_{i,(k)} E_{i,j}Q_{j,(k)}E_{i,j}^\top\right)x_{S_j,k}^\top Q_{j,(k)}\beta_{S_j}\beta_{S_j\setminus S_i}^\top x_{S_j\setminus S_i,k}}{\left(1+\frac{1}{n_i}\mathrm{Tr}\left( Q_{i,(k)}\right)\right)\left(1 +  \frac{1}{n_j}\mathrm{Tr}\left( Q_{j,(k)}\right)\right)}\right|^q\right] = O\left(\frac{1}{n^{q/2}}\right),
    \end{align}
      by combining Lemmas \ref{bai_ineqaulity1}, \ref{bai_inequality2}, \ref{rubio_inequality1} with the Cauchy-Schwartz inequality. On the other hand, if $k \in T_i \setminus T_j$, we can obtain
    \begin{align}
        \delta_{2,k} &= -\lambda x_{S_i,k}^\top Q_i E_{i,j}Q_j \beta_{S_j}\beta_{S_j\setminus S_i}^\top x_{S_j\setminus S_i,k} \nonumber\\
        &=  -\lambda \frac{x_{S_i,k}^\top Q_{i,(k)} E_{i,j}Q_{j} \beta_{S_j}\beta_{S_j\setminus S_i}^\top x_{S_j\setminus S_i,k}}{\left(1+\frac{1}{n}x_{S_i,k}^\top Q_{i,(k)}x_{S_i,k}\right)}
    \end{align}
    and 
    \begin{align}
        \mathbbm{E}\left[\left| \lambda \frac{x_{S_i,k}^\top Q_{i,(k)} E_{i,j}Q_{j} \beta_{S_j}\beta_{S_j\setminus S_i}^\top x_{S_j\setminus S_i,k}}{\left(1+\frac{1}{n_i}x_{S_i,k}^\top Q_{i,(k)}x_{S_i,k}\right)} - \lambda \frac{x_{S_i,k}^\top Q_{i,(k)} E_{i,j}Q_{j} \beta_{S_j}\beta_{S_j\setminus S_i}^\top x_{S_j\setminus S_i,k}}{\left(1+\frac{1}{n_i}\mathrm{Tr}\left( Q_{i,(k)}\right)\right)} \right|^q\right] = O\left(\frac{1}{n^{q/2}}\right).
    \end{align}
    Thus we obtain
    \begin{align}
        \Delta_2 \to
        \begin{cases}
            0 \quad &\gamma_j < \eta_j,\\
            \frac{\eta_{ij}\gamma_{ij}(\gamma_{j}-\eta_j)}{\gamma_{j}(\eta_i\gamma_j-\eta_{ij}\gamma_{ij})}r^2(\kappa_j -\kappa_{ij}) \quad &\gamma_i <\eta_i, \eta_j < \gamma_j,\\
            \frac{\eta_{ij}\gamma_{ij}(\gamma_{j}-\eta_j)}{\gamma_{j}(\gamma_i\gamma_j-\eta_{ij}\gamma_{ij})}r^2(\kappa_j -\kappa_{ij}) \quad &\eta_i < \gamma_i, \eta_j<\gamma_j, 
        \end{cases} 
    \end{align}
    with Lemma \ref{important_lemma}, Lemma \ref{rubio_convergence1}, Arz\`ela-Ascoli theorem, and the bounded convergence theorem. $\Delta_3 \to 0$ can be shown similarly.
    The term $\Delta_4$ can be rewritten as:
    \begin{align*}
        \Delta_4 &= \mathbbm{E}\left[\beta_{S_j\setminus S_i}^\top X^{T_i\top}_{S_j\setminus S_i}  X^{T_i}_{S_i} (X^{T_i\top}_{S_i}  X^{T_i}_{S_i})^+E_{i,j}(X^{T_j\top}_{S_j}  X^{T_j}_{S_j})^+ X^{T_j\top}_{S_j}  X^{T_j}_{S_i\setminus S_j}\beta_{S_i\setminus S_j}\right]\\
        &= \frac{1}{n_{ij}}\sum_{k\in T_i\cap T_j}\frac{n_{ij}}{n_in_j}\left(\mathbbm{E}\left[ x_{S_i,k}^\top \left(\frac{1}{n_i}X^{T_i\top}_{S_i}  X^{T_i}_{S_i}\right)^+E_{i,j}\left(\frac{1}{n_j}X^{T_j\top}_{S_j}  X^{T_j}_{S_j}\right)^+ X^{T_j\top}_{S_j, (k)}  X^{T_j}_{S_i\setminus S_j, (k)}\beta_{S_i\setminus S_j}\beta_{S_j\setminus S_i}^\top x_{S_j\setminus S_i,k}\right] \right.\\
        &\quad \left.+ \mathbbm{E}\left[ x_{S_i,k}^\top\left(\frac{1}{n_i}X_{S_i}^\top X_{S_i}\right)^+E_{i,j}\left(\frac{1}{n_j}X^{T_j\top}_{S_j}  X^{T_j}_{S_j}\right)^+ x_{S_j,k} x_{S_i\setminus S_j, k}^\top\beta_{S_i\setminus S_j}\beta_{S_j\setminus S_i}^\top x_{S_j\setminus S_i,k}\right]\right)\\
        & \quad + \frac{1}{n_{ij}}\sum_{k\in T_i\setminus T_j}\frac{n_{ij}}{n_in_j}\mathbbm{E}\left[ x_{S_i,k}^\top \left(\frac{1}{n_i}X^{T_i\top}_{S_i}  X^{T_i}_{S_i}\right)^+E_{i,j}\left(\frac{1}{n_j}X^{T_j\top}_{S_j}  X^{T_j}_{S_j}\right)^+ X^{T_j\top}_{S_j}  X^{T_j}_{S_i\setminus S_j}\beta_{S_i\setminus S_j}\beta_{S_j\setminus S_i}^\top x_{S_j\setminus S_i,k}\right] \\
        &= \Delta_{41} + \Delta_{42} + \Delta_{43},
    \end{align*}
    For $\Delta_4\to 0$, we consider the following: 
    \begin{align}
        \delta_{41,k}(\lambda) = \delta_{41,k} = \frac{n_{ij}}{n_in_j} x_{S_i,k}^\top Q_iE_{i,j}Q_j X_{S_j, (k)}^\top X_{S_i\setminus S_j, (k)}\beta_{S_i\setminus S_j}\beta_{S_j\setminus S_i}^\top x_{S_j\setminus S_i,k}
    \end{align}
    \begin{align}
        \delta_{42,k}(\lambda) = \delta_{42,k} = \frac{n_{ij}}{n_in_j} x_{S_i,k}^\top Q_iE_{i,j}Q_j x_{S_j,k} x_{S_i\setminus S_j, k}^\top\beta_{S_i\setminus S_j}\beta_{S_j\setminus S_i}^\top x_{S_j\setminus S_i,k}
    \end{align}
    for $\lambda>0$. We assume that $k \in T_i\cap T_j$.
    Note that
    \begin{align*}
    \mathbbm{E}\left[\lim_{\lambda\to 0}\sum_{k\in T_i\cap T_j}\delta_{41,k}\right] = \Delta_{41}
    \end{align*}
    \begin{align*}
    \mathbbm{E}\left[\lim_{\lambda\to 0}\sum_{k\in T_i\cap T_j}\delta_{42,k}\right] = \Delta_{42},
    \end{align*}
    and $\sum_{k\in T_i\cap T_j}\delta_{41,k}$ and $\sum_{k\in T_i\cap T_j}\delta_{42,k}$ are uniformly bounded on $\lambda\in [0,C]$ for a positive real number $C$.
    First, we demonstrate that $\Delta_{42}\to 0$. Using the technique in the proof of Theorem \ref{main_result1}, we obtain
    \begin{align}
          &\mathbbm{E}\left[\left|\delta_{42,k}-\frac{\frac{n_{ij}}{n_in_j}\mathrm{Tr}\left(Q_{i,(k)}E_{i,j}Q_{j,(k)}E_{i,j}^\top\right)}{(1+\frac{1}{n_i}\mathrm{Tr}\left(Q_{i,(k)}\right))(1+\frac{1}{n_j}\mathrm{Tr}\left(Q_{j,(k)}\right))}  x_{S_i\setminus S_j, k}^\top\beta_{S_i\setminus S_j}\beta_{S_j\setminus S_i}^\top x_{S_j\setminus S_i,k}  \right|^q\right]\nonumber\\ &=\mathbbm{E}\left[\left| \left(\frac{n_{ij}}{n_in_j} x_{S_i,k}^\top Q_iE_{i,j}Q_j x_{S_j,k} -\frac{\frac{n_{ij}}{n_in_j}\mathrm{Tr}\left(Q_{i,(k)}E_{i,j}Q_{j,(k)}E_{i,j}^\top\right)}{(1+\frac{1}{n_i}\mathrm{Tr}\left(Q_{i,(k)}\right))(1+\frac{1}{n_j}\mathrm{Tr}\left(Q_{j,(k)}\right))} \right) x_{S_i\setminus S_j, k}^\top\beta_{S_i\setminus S_j}\beta_{S_j\setminus S_i}^\top x_{S_j\setminus S_i,k}\right|^q\right]\nonumber\\
          &\le\mathbbm{E}^{1/2}\left[\left|\frac{n_{ij}}{n_in_j} x_{S_i,k}^\top Q_iE_{i,j}Q_j x_{S_j,k} -\frac{\frac{n_{ij}}{n_in_j}\mathrm{Tr}\left(Q_{i,(k)}E_{i,j}Q_{j,(k)}E_{i,j}^\top\right)}{(1+\frac{1}{n_i}\mathrm{Tr}\left(Q_{i,(k)}\right))(1+\frac{1}{n_j}\mathrm{Tr}\left(Q_{j,(k)}\right))} \right|^{2q}\right] \nonumber\\
          &\quad\quad\quad\quad\times\mathbbm{E}^{1/2}\left[\left|x_{S_i\setminus S_j, k}^\top\beta_{S_i\setminus S_j}\beta_{S_j\setminus S_i}^\top x_{S_j\setminus S_i,k}\right|^{2q}\right] \label{eqn7.1}
    \end{align}
    The first term in (\ref{eqn7.1}) can be written as
    \begin{align}
        &\mathbbm{E}^{1/2}\left[\left| \frac{n_{ij}}{n_in_j} x_{S_i,k}^\top Q_iE_{i,j}Q_j x_{S_j,k}-\frac{\frac{n_{ij}}{n_in_j}\mathrm{Tr}\left(Q_{i,(k)}E_{i,j}Q_{j,(k)}E_{i,j}^\top\right)}{(1+\frac{1}{n_i}\mathrm{Tr}\left(Q_{i,(k)}\right))(1+\frac{1}{n_j}\mathrm{Tr}\left(Q_{j,(k)}\right))}\right|^{2q}\right]\nonumber\\
        &=\mathbbm{E}^{1/2}\left[\left|  \frac{\frac{n_{ij}}{n_in_j}x_{S_i,k}^\top Q_{i,(k)}E_{i,j}Q_{j,(k)} x_{S_j,k}}{(1+\frac{1}{n_i}x_{S_i,k}^\top Q_{i,(k)}x_{S_i,k})(1+\frac{1}{n_j}x_{S_j,k}^\top Q_{j,(k)}x_{S_j,k})}-\frac{\frac{n_{ij}}{n_in_j}\mathrm{Tr}\left(Q_{i,(k)}E_{i,j}Q_{j,(k)}E_{i,j}^\top\right)}{(1+\frac{1}{n_i}\mathrm{Tr}\left(Q_{i,(k)}\right))(1+\frac{1}{n_j}\mathrm{Tr}\left(Q_{j,(k)}\right))}\right|^{2q}\right]\nonumber\\
        &=O\left(\frac{1}{n^{q/2}}\right)\label{eqn6.1}
    \end{align}
     In (\ref{eqn6.1}), we use Lemma \ref{bai_ineqaulity1} and Lemma \ref{bai_inequality2}. Hence, from Lemma \ref{rubio_convergence1} and Remark \ref{remark_for_rubio}, we obtain
     \begin{align}
       \frac{1}{n_{ij}}\mathbbm{E}\left[\sum_{k\in T_i\cap T_j} \frac{\frac{n_{ij}}{n_in_j}\mathrm{Tr}\left(Q_{i,(k)}E_{i,j}Q_{j,(k)}E_{i,j}^\top\right)}{(1+\frac{1}{n_i}\mathrm{Tr}\left(Q_{i,(k)}\right))(1+\frac{1}{n_j}\mathrm{Tr}\left(Q_{j,(k)}\right))}       x_{S_i\setminus S_j, k}^\top\beta_{S_i\setminus S_j}\beta_{S_j\setminus S_i}^\top x_{S_j\setminus S_i,k} \right] \to 0.
    \end{align}
     $\Delta_{42}\to 0$ follows by applying Arz\`ela-Ascoli theorem and bounded convergence theorem. Next, we consider the term $\delta_{41,k}$. Observing
    \begin{align}
        \delta_{41,k} &=  \frac{n_{ij}}{n_in_j} x_{S_i,k}^\top Q_iE_{i,j}Q_j X^{T_j\top}_{S_j, (k)}  X^{T_j}_{S_i\setminus S_j, (k)}\beta_{S_i\setminus S_j}\beta_{S_j\setminus S_i}^\top x_{S_j\setminus S_i,k} \nonumber\\
        &= \frac{n_{ij}}{n_in_j} \frac{x_{S_i,k}^\top Q_{i,(k)}E_{i,j}Q_{j,(k)} X^{T_j\top}_{S_j, (k)}  X^{T_j}_{S_i\setminus S_j, (k)}\beta_{S_i\setminus S_j}\beta_{S_j\setminus S_i}^\top x_{S_j\setminus S_i,k} }{(1+\frac{1}{n_i}x_{S_i,k}^\top Q_{i,(k)}x_{S_i,k})} \nonumber\\
        &\quad \quad- \frac{n_{ij}}{n_in_j^2}\frac{x_{S_i,k}^\top Q_{i,(k)}E_{i,j}Q_{j,(k)}x_{S_j,k} x_{S_j,k}^\top Q_{j,(k)}X^{T_j\top}_{S_j, (k)}  X^{T_j}_{S_i\setminus S_j, (k)}\beta_{S_i\setminus S_j}\beta_{S_j\setminus S_i}^\top x_{S_j\setminus S_i,k} }{(1+\frac{1}{n_i}x_{S_i,k}^\top Q_{i,(k)}x_{S_i,k})(1+\frac{1}{n_j}x_{S_j,k}^\top Q_{j,(k)}x_{S_j,k})}, \label{eqn7.2}
    \end{align}
    the expectation of the first term in (\ref{eqn7.2}) converges to $0$ by using the procedure similar to $\delta_{42,k}$. For the second term of (\ref{eqn7.2}), we first obtain
    \begin{align}
    &\mathbbm{E}\left[\left|\delta_{41,k}\right|^q\right] \nonumber\\
    &= \mathbbm{E}\left[\left| \frac{n_{ij}}{n_in_j^2}\frac{x_{S_i,k}^\top Q_{i,(k)}E_{i,j}Q_{j,(k)}x_{S_j,k} x_{S_j,k}^\top Q_{j,(k)} X^{T_j\top}_{S_j, (k)}  X^{T_j}_{S_i\setminus S_j, (k)}\beta_{S_i\setminus S_j}\beta_{S_j\setminus S_i}^\top x_{S_j\setminus S_i,k} }{(1+\frac{1}{n_i}x_{S_i,k}^\top Q_{i,(k)}x_{S_i,k})(1+\frac{1}{n_j}x_{S_j,k}^\top Q_{j,(k)}x_{S_j,k})}\right.\right.\nonumber\\
    &\quad \quad\left.\left.- \frac{n_{ij}}{n_in_j^2}\frac{ \mathrm{Tr}\left(Q_{i,(k)}E_{i,j}Q_{j,(k)}E_{i,j}^\top\right) x_{S_j,k}^\top Q_{j,(k)} X^{T_j\top}_{S_j, (k)}  X^{T_j}_{S_i\setminus S_j, (k)}\beta_{S_i\setminus S_j}\beta_{S_j\setminus S_i}^\top x_{S_j\setminus S_i,k} }{(1+\frac{1}{n_i}\mathrm{Tr} (Q_{i,(k)}))(1+\frac{1}{n_j}\mathrm{Tr}(Q_{j,(k)}))}\right|^q\right]
    =O\left(\frac{1}{n^{q/2}}\right).
    \end{align}
     from Lemmas \ref{bai_ineqaulity1} and \ref{rubio_convergence1} along with the Cauchy-Schwartz inequality and the fact 
     \begin{align}
         &\mathbbm{E}\left[\left|\frac{1}{n_j}x_{S_j,k}^\top Q_{j,(k)} X^{T_j\top}_{S_j, (k)}  X^{T_j}_{S_i\setminus S_j, (k)}\beta_{S_i\setminus S_j}\beta_{S_j\setminus S_i}^\top x_{S_j\setminus S_i,k}\right|^{2q}\right]\nonumber\\
         &\le C_q \|\beta_{S_j\setminus S_i}\|^{q} \mathbbm{E}\left[\left|\frac{1}{n_j^2}\beta_{S_i\setminus S_j}^\top X^{T_j\top}_{S_i\setminus S_j, (k)}  X^{T_j}_{S_j, (k)} Q_{j,(k)}^2 X^{T_j\top}_{S_j, (k)}  X^{T_j}_{S_i\setminus S_j, (k)}\beta_{S_i\setminus S_j}\right|^{q}\right]\label{eqn11.1}\\
         &\le \frac{1}{\lambda^2}D_q \|\beta_{S_j\setminus S_i}\|^{q} \mathbbm{E}\left[\left|\frac{1}{n_j}\beta_{S_i\setminus S_j}^\top  X^{T_j\top}_{S_i\setminus S_j, (k)}  X^{T_j}_{S_i\setminus S_j, (k)}\beta_{S_i\setminus S_j}\right|^{q}\right]\label{eqn11.2}\\
         &= \frac{1}{\lambda^2}D_q \|\beta_{S_j\setminus S_i}\|^{q} \mathbbm{E}\left[\left| \frac{1}{n_j}\sum_{l\in T_j, l\neq k}x_{S_i\setminus S_j, l}^\top\beta_{S_i\setminus S_j}  \beta_{S_i\setminus S_j}^\top x_{S_i\setminus S_j, l}\right|^{q}\right]\nonumber\\
         &\le \frac{1}{\lambda^2}D_q \|\beta_{S_j\setminus S_i}\|^{q} \frac{1}{n_j}\sum_{l\in T_j, l\neq k}\mathbbm{E}\left[\left| x_{S_i\setminus S_j, l}^\top\beta_{S_i\setminus S_j}  \beta_{S_i\setminus S_j}^\top x_{S_i\setminus S_j, l}\right|^{q}\right] \label{eqn11.3}\\
         &= O\left(1\right) \label{eqn11.4}
     \end{align}
    where $C_q,D_q$ are constants that depend only on $q$. Here, we use Lemma \ref{rubio_inequality1} in (\ref{eqn11.1}), the fact that $\|Q_{j,(k)}\|_2 < 1/\lambda$ and $\|X^{T_j}_{S_j, (k)}/\sqrt{n_j}\|$ is upper-bounded by a constant for all large $n$ (e.g. see Theorem 1 and its Corollary of \cite{bai1998no}) in (\ref{eqn11.2}), the fact that the mapping $x\mapsto |x|^q$ is convex in (\ref{eqn11.3}), and Lemma \ref{rubio_inequality1} in (\ref{eqn11.4}). In addition, from Lemma \ref{rubio_convergence1} and Remark \ref{remark_for_rubio}, we obtain
    \begin{align}
        &\frac{1}{n_{ij}}\sum_{k\in T_i\cap T_j}\mathbbm{E}\left[\frac{n_{ij}}{n_in_j^2}\frac{ \mathrm{Tr}\left(Q_{i,(k)}E_{i,j}Q_{j,(k)}E_{i,j}^\top\right) x_{S_j,k}^\top Q_{j,(k)} X^{T_j\top}_{S_j, (k)}  X^{T_j}_{S_i\setminus S_j, (k)}\beta_{S_i\setminus S_j}\beta_{S_j\setminus S_i}^\top x_{S_j\setminus S_i,k} }{(1+\frac{1}{n_i}\mathrm{Tr} (Q_{i,(k)}))(1+\frac{1}{n_j}\mathrm{Tr}(Q_{j,(k)}))}\right] \nonumber\\
        &\quad\quad -\frac{1}{n_{ij}}\sum_{k\in T_i\cap T_j}\mathbbm{E}\left[\frac{n_{ij}}{n_in_j^2}\frac{ \mathrm{Tr}\left(Q_{i,(k)}E_{i,j}Q_{j,(k)}E_{i,j}^\top\right) \mathrm{Tr}\left( Q_{j,(k)} X^{T_j\top}_{S_j, (k)}  X^{T_j}_{S_i\setminus S_j, (k)}\beta_{S_i\setminus S_j}\beta_{S_j\setminus S_i}^\top F_{j,i}\right) }{(1+\frac{1}{n_i}\mathrm{Tr} (Q_{i,(k)}))(1+\frac{1}{n_j}\mathrm{Tr}(Q_{j,(k)}))} \right] \to 0
    \end{align}
    Moreover, we can obtain 
    \begin{align}
         &\mathbbm{E}\left[\left| \frac{n_{ij}}{n_in_j^2}\frac{ \mathrm{Tr}\left(Q_{i,(k)}E_{i,j}Q_{j,(k)}E_{i,j}^\top\right) \mathrm{Tr}\left(X_{S_j, (k)}^\top X_{S_i\setminus S_j, (k)}\beta_{S_i\setminus S_j}\beta_{S_j\setminus S_i}^\top F_{j,i}\right) }{(1+\frac{1}{n}\mathrm{Tr} (Q_{i,(k)}))(1+\frac{1}{n}\mathrm{Tr}(Q_{j,(k)}))}\right|^q\right]
        =O\left(\frac{1}{n^{q/2}}\right).
    \end{align}
    from Remark \ref{remark_for_rubio}. Therefore, $\Delta_{41}\to0$ together with Arz\`ela-Ascoli theorem
    and Lemma \ref{lem_for_sum_of_rv}. $\Delta_{43}\to 0$ can be proven in a method analogous to the proof of $\Delta_{41}\to 0$.

    In terms of $\Delta_6, \Delta_7$, when $\gamma_j < \eta_j$,
    \begin{align*}
        \Delta_6 &= \mathbbm{E}\left[(\beta_{S_j} - \hat\beta_{S_j})^\top F_{j,i} \beta_{S_j \setminus S_i}\right]\\
        &= \mathbbm{E}\left[\beta_{S_j}^\top(\mathrm{I} - (X_{S_j}^\top X_{S_j})^+ X_{S_j}^\top X_{S_j}) F_{j,i} \beta_{S_j \setminus S_i}\right]\\
        &=0,
    \end{align*}
    and when $\gamma_j > \eta_j$
    \begin{align*}
        \Delta_6 &= \mathbbm{E}\left[(\beta_{S_j} - \hat\beta_{S_j})^\top F_{j,i} \beta_{S_j \setminus S_i}\right]\\
        &= \mathbbm{E}\left[\beta_{S_j}^\top(\mathrm{I} - (X_{S_j}^\top X_{S_j})^+ X_{S_j}^\top X_{S_j}) F_{j,i} \beta_{S_j \setminus S_i}\right] \\
        &\to \left(1-\frac{\eta_j}{\gamma_j}\right)r^2(\kappa_j -\kappa_{ij}).
    \end{align*}
    as $n,p\to\infty$, which can be shown, for example, by following the same procedure as in the proof of Theorem 1 in \cite{hastie2022surprises}. The proof of the convergence of $\Delta_7$ remains the same. This completes the proof.
\end{proof}

\subsection{Some widely-known results}
In this subsection, we introduce some widely-known results that are frequently used in the proofs of theorems. Throughout the following, $K_p$ is a constant that may have different values for each appearance.

\begin{lemma}
\label{lem_for_sum_of_rv}
Let $\{y_n^{(N)}, 1\le n\le N\}$ denote a collection of random variables such that
\begin{align*}
    \max_{1\le n \le N} \mathbbm{E}\left[|y_n^{(N)}|^p\right] \le \frac{K_p}{N^{1+\delta}},
\end{align*}
for some constants $p \ge 1$, $\delta >0$ and $K_p$  depending on $p$ but not on $N$.
Then, almost surely as $N\to\infty$,
\begin{align*}
    \frac{1}{N}\sum_{n=1}^N |y_n^{(N)}|\to 0.
\end{align*}
\end{lemma}
\begin{proof}
    Since, the mapping $x \to |x|^p$ is a convex function, 
    \begin{align}
        \mathbbm{E}\left[\left|\frac{1}{N}\sum_{n=1}^N|y_n^{(N)}|\right|^p\right] &=  \mathbbm{E}\left[\left|\frac{1}{N}\sum_{n=1}^N|y_n^{(N)}|\right|^p\right]\\
        &\le  \frac{1}{N}\sum_{n=1}^N \mathbbm{E}\left[\left|y_n^{(N)}\right|^p\right] \nonumber\\
        &\le \frac{K_p}{N^{1+\delta}},
    \end{align}
    for large $N$, where the last inequality is due to the assumption.
    Hence, from Borel–Cantelli lemma, we can show the result.
\end{proof}

The following two lemmas are adapted from \cite{rubio2011spectral}.

\begin{lemma}[\cite{rubio2011spectral}, Lemma 3]
\label{rubio_inequality1}
Let $\xi\in\mathbbm{C}^M$ denote a random vector with i.i.d. entries having mean zero and variance one, and $C\in\mathbbm{C}^{M\times M}$ an arbitrary nonrandom matrix. Then, for any $p\ge 2$,
\begin{align*}
\mathbbm{E}\left[|\xi^\top C\xi|^p\right]\le \|C\|_{\mathrm{Tr}}^p\left(K_{1,p} + K_{2,p}\mathbbm{E}\left[|\zeta|^{2p}\right]\right),    
\end{align*}
where $\zeta$ denotes a particular entry of $\xi$ and the constants $K_{1,p}$ and $K_{2,p}$ do not depend on $M$, the entries of $C$, nor the distribution of $\zeta$.
\end{lemma}

\begin{lemma}[\cite{rubio2011spectral}, Lemma 4]
\label{rubio_convergence1}
Let $\mathcal{U}=\{\xi_n\in\mathbbm{C}^M, 1\le n \le N\}$ denote a collection of i.i.d. random vectors defined as in Lemma \ref{rubio_inequality1}, and whose entries are assumed to have finite $4+\epsilon$ moment, $\epsilon >0$. Furthermore, consider a collection of random matrices $\{C_{(n)}\in\mathbbm{C}^{M\times M}, 1\le n \le N\}$ such that, for each $n$, $C_{(n)}$ may depend on all the elements of $\mathcal{U}$ except for $\xi_n$, and $\|C_{(n)}\|_{\mathrm{Tr}}$ is uniformly bounded for all $M$. Then, almost surely as $N \to \infty$,
\begin{align*}
    \left|\frac{1}{N}\sum_{n=1}^N \left(\xi^\top _n C_{(n)} \xi_n -\mathrm{Tr}(C_{(n)}) \right) \right| \to 0.
\end{align*}
\end{lemma}

\begin{remark}
\label{remark_for_rubio}
    As we can see in the proof of the above Lemma in \cite{rubio2011spectral}, 
\begin{align*}
    \mathbbm{E}\left[\left|\frac{1}{N}\sum_{n=1}^N \left(\xi^\top _n C_{(n)} \xi_n -\mathrm{Tr}(C_{(n)}) \right) \right|^p\right] = O\left(\frac{1}{N^{p/2}} + \frac{1}{N^{p-1}}\right)
\end{align*}
\end{remark}

The next Lemma is adapted from \cite{bai1998no}.
\begin{lemma}[\cite{bai1998no}, Lemma 2.7]
\label{bai_ineqaulity1}
Let $\xi\in\mathbbm{C}^M$ denote a random vector with i.i.d. entries having mean zero and variance one, and $C\in\mathbbm{C}^{M\times M}$ an arbitrary nonrandom matrix. Then, for any $p\ge 2$,
\begin{align*}
    \mathbbm{E}\left[|\xi^\top C\xi - \mathrm{Tr}(C)|^p\right] \le K_p\left\{\left(\mathbbm{E}\left[|\xi|^4\right]\mathrm{Tr}(CC^\top )\right)^{p/2}+\mathbbm{E}\left[|\xi|^{2p}\right]\mathrm{Tr}\left[(CC^\top )^{p/2}\right]\right\},
\end{align*}
where $\zeta$ denotes a particular entry of $\xi$ and the constant does not depend on $M$, the entries of $C$, nor the distribution of $\zeta$.
\end{lemma}
 The following Lemma is the real number version of Lemma 2.10 of \cite{bai1998no}, which can be easily obtained.
\begin{lemma}
\label{bai_inequality2}
Consider two $N\times N$ matrices $B_1$ and $B_2$, with $B_2$ being Hermitian, and $\tau\in \mathbbm{R}$, $q\in\mathbbm{C}^N$. Then, for each $\lambda>0$,
\begin{align*}
    \left| \mathrm{Tr}\left( B_1 ( (B_2+\lambda\mathrm{I}_M)^{-1} \right.\right. & \left. \left.- (B_2+\tau qq^\top  +\lambda\mathrm{I}_M)^{-1}  ) \right) \right| \\
    &= \left|\frac{ \tau q^\top ( B_2+\lambda\mathrm{I}_M )^{-1}B_1 ( B_2+\lambda\mathrm{I}_M )^{-1}q }{ 1+\tau q^\top (B_2+\lambda \mathrm{I}_M)^{-1}q}\right| \le \frac{\|B_1\|_2}{\lambda}.
\end{align*}
\end{lemma}

\subsection{Auxiliary lemmas}
This section proves any relevant theorems whose proofs were omitted previously. First, we present the primary results. The following lemma is used in the proofs of Theorems \ref{main_result1} and \ref{main_result2}:

\begin{lemma}
\label{important_lemma}
    Let $Q_1 \coloneqq Q_1(\lambda)\coloneqq \left(n_1^{-1}Z_1^\top Z_1 +\lambda \mathrm{I}_{p_1}\right)^{-1}$ and $Q_2\coloneqq Q_2(\lambda)\coloneqq \left(n_2^{-1}Z_2^\top Z_2 +\lambda \mathrm{I}_{p_2}\right)^{-1} $, where $Z_1$ and $Z_2$ are $n_1\times p_1$ and $n_2\times p_2$ matrices whose elements are i.i.d. random variables of mean $0$ and variance $1$, respectively, and the first $n_{12} \times p_{12}$ principal submatrices of these two matrices are assumed to consist of common elements and all other parts are assumed to be independent. Suppose moreover that $0<\liminf_{n\to\infty}p_1/n \le \limsup_{n\to\infty}p_1/n < \infty  $, $0<\liminf_{n\to\infty}p_2/n \le \limsup_{n\to\infty}p_2/n < \infty  $, $0<\liminf_{n\to\infty}p_{12}/n \le \limsup_{n\to\infty}p_{12}/n < \infty  $ $0<\liminf_{n\to\infty}n_{12}/n \le \limsup_{n\to\infty}n_{12}/n < \infty $, $n_1\neq p_1$ and $n_2\neq p_2$. Furthermore, for $l=1,2$, define $e_l$ as follows,
    \begin{align*}
        e_l(-\lambda) = e_l =  \frac{-1+\frac{p_l}{n_l} -\lambda+\sqrt{(1-\frac{p_l}{n_l}+\lambda)^2+4p_l \lambda/n_l}}{2p_l \lambda/n_l}
    \end{align*}
    which is the Stieltjes transform of Marchenko–Pastur distribution.
    Then, we have that for any $\lambda > 0$,
    \begin{align*}
        \mathrm{Tr} \left(Q_1\Theta_1 Q_2 \Theta_2\right) - e_1(-\lambda)e_2(-\lambda)\mathrm{Tr}\left\{\Theta_1 \Theta_2\right\} 
        - e_1(-\lambda)e_2(-\lambda)\frac{\mathrm{Tr}\left\{ E_{12}\Theta_2\right\} m_{12}^* (\lambda)}{\left(1+\frac{p_{1}}{n_1}e_1(-\lambda)\right)\left(1+\frac{p_{2}}{n_2}e_2(-\lambda)\right)} \to 0,
    \end{align*}
    almost surely, where $\Theta_1$, $\Theta_2$ are $p_1\times p_2$, $p_2\times p_1$ matrix such that $\|\Theta_1\|_2 < \infty $ and $\|\Theta_2\|_{\mathrm{Tr}}<\infty$, respectively. In addition, $m_{12}^* (\lambda)$ satisfies
    \begin{align*}
        m_{12}^* (\lambda) = \frac{n_{12}}{n_1n_2}e_1(-\lambda)e_2(-\lambda)\mathrm{Tr}\left\{\Theta_1 E_{12} \right\} 
        + e_1(-\lambda)e_2(-\lambda)\frac{\frac{n_{12}p_{12}}{n_1n_2} m_{12}^*(z)}{\left(1+\frac{p_1}{n_1}e_1(-\lambda)\right)\left(1+\frac{p_2}{n_2}e_2(-\lambda)\right)}, 
    \end{align*}
    where  $E_{12} = \mathbbm{E}\left[z_{11} z_{21}^\top\right]$, where $z_{i1}$ is the first row vector of $Z_i$ for $i=1,2$. 
\end{lemma}

\begin{proof}
For notational convenience, we will use the following definitions:
\begin{align*}
Z_l=\left(z_{l1},...,z_{ln}\right)^\top&, \quad x_l(e) = \frac{1}{1+ p_l/n e}, \\
S_1 = n_1^{-1}Z_1^\top Z_1&, \quad S_2 = n_2^{-1}Z_2^\top Z_2,\\ S_{1(k)} = S_1 - \frac{1}{n_1}z_{1k}z_{1k}^\top&, \quad S_{2(k)} = S_2 - \frac{1}{n_2}z_{2k}z_{2k}^\top, \\
Q_{1(k)} = (S_{1(k)} +\lambda I)^{-1}&, \quad Q_{2(k)} = (S_{2(k)} +\lambda I)^{-1},\\
P_1(e) = \frac{1}{x_1(e)+\lambda}I&, P_2(e) = \frac{1}{x_2(e)+\lambda}I.
\end{align*}
for $l=1,2$ and a positive real number $e$. Note that, it can be easily seen that $\|Q_1\|_2$, $\|Q_2\|_2$, $\|Q_{1(k)}\|_2$, $\|Q_{2(k)}\|_2$, $\|P_1(e)\|_2$ and $\|P_2(e)\|_2$ are upper-bounded by $1/\lambda$. Furthermore, let $\hat{e}_{1}(-\lambda)=\frac{1}{p_1}\mathrm{Tr}\left(Q_1(\lambda)\right)$, $\hat{e}_{2}(-\lambda)=\frac{1}{p_2}\mathrm{Tr}\left(Q_2(\lambda)\right)$, $\hat{e}_1^{(k)}(-\lambda)=\frac{1}{p_1}\mathrm{Tr}(Q_{1(k)(\lambda)})$ and $\hat{e}_2^{(k)}(-\lambda) =\frac{1}{p_2}\mathrm{Tr}(Q_{2(k)}(\lambda))$. Now, consider the identities

\begin{align*}
    Q_1^{-1}(\lambda) &= S_1+\lambda\mathrm{I} = x_1(\hat{e}_1)I +\lambda\mathrm{I} + \frac{1}{n_1} Z_1^\top Z_1 -  x_1(\hat{e}_1)I,\\
    Q_2^{-1}(\lambda) &= S_2+\lambda\mathrm{I} = x_2(\hat{e}_2)I +\lambda\mathrm{I} + \frac{1}{n_2} Z_2^\top Z_2 -  x_2(\hat{e}_2)\mathrm{I},
\end{align*}

By using the resolvent identity, 
\begin{align}
    P_1(\hat{e}_1) - Q_1(\lambda) = P_1(\hat{e}_1)\left(\frac{1}{n_1} Z_1^\top Z_1 -  x_1(\hat{e}_1)\right)Q_1(\lambda),\label{eqn_resolvent1}\\
    P_2(\hat{e}_2) - Q_2(\lambda) = P_2(\hat{e}_2)\left(\frac{1}{n_2} Z_2^\top Z_2 -  x_2(\hat{e}_2)\right)Q_2(\lambda).\label{eqn_resolvent2}
\end{align}

Using (\ref{eqn_resolvent1}), we now consider the following identities:
\begin{align}
    \left(\frac{1}{n_1} Z_1^\top Z_1\right)Q_1\Theta_1Q_2\Theta_2 &= \frac{1}{n_1}\sum_{k=1}^{n_1} z_{1k}z_{1k}^\top Q_1\Theta_1Q_2\Theta_2\nonumber\\
    &=\frac{1}{n_1} \sum_{k=1}^{n_1} \frac{1}{1+\frac{1}{n_1}z_{1k}^\top Q_{1(k)}z_{1k}}z_{1k}z_{1k}^\top Q_{1(k)}\Theta_1Q_2\Theta_2 \label{main_eqn1}\\
    &= \frac{1}{n_1} \sum_{k=1}^{n_{12}} \frac{1}{1+\frac{1}{n_1}z_{1k}^\top Q_{1(k)}z_{1k}}z_{1k}z_{1k}^\top Q_{1(k)}\Theta_1Q_{2(k)}\Theta_2 \nonumber\\ 
    &\quad +\frac{1}{n_1} \sum_{k=n_{12}+1}^{n_{1}} \frac{1}{1+\frac{1}{n_1}z_{1k}^\top Q_{1(k)}z_{1k}}z_{1k}z_{1k}^\top Q_{1(k)}\Theta_1Q_{2}\Theta_2 \nonumber\\&\quad \quad - \frac{1}{n_1}\sum_{k=1}^{n_{12}} \frac{z_{1k}^\top  Q_{1(k)}\Theta_1Q_{2(k)}z_{2k}}{\left(1+\frac{1}{n_1}z_{1k}^\top Q_{1(k)}z_{1k}\right)\left(1+\frac{1}{n_2}z_{2k}^\top Q_{2(k)}z_{2k}\right)} \frac{1}{n_2}z_{1k}z_{2k}^\top Q_{2(k)} \Theta_2, \label{main_eqn2}\\
    &= \frac{1}{n_1} \sum_{k=1}^{n_{1}} \frac{1}{1+\frac{1}{n_1}z_{1k}^\top Q_{1(k)}z_{1k}}z_{1k}z_{1k}^\top Q_{1(k)}\Theta_1\Tilde{Q}_{2(k)}\Theta_2 \nonumber\\ 
    &\quad \quad - \frac{1}{n_1}\sum_{k=1}^{n_{12}} \frac{z_{1k}^\top  Q_{1(k)}\Theta_1Q_{2(k)}z_{2k}}{\left(1+\frac{1}{n_1}z_{1k}^\top Q_{1(k)}z_{1k}\right)\left(1+\frac{1}{n_2}z_{2k}^\top Q_{2(k)}z_{2k}\right)} \frac{1}{n_2}z_{1k}z_{2k}^\top Q_{2(k)} \Theta_2, 
\end{align}
where, $\Tilde{Q}_{2(k)} = \begin{cases}
    Q_{2(k)} &\quad k\le n_{12},\\
    Q_{2} &\quad k> n_{12},
\end{cases}$. In (\ref{main_eqn1}),(\ref{main_eqn2}), we use the Sherman-Morrison-Woodbury identity, i.e.,
\begin{align}
    Q_1(\lambda) = Q_{1(k)} - \frac{\frac{1}{n_1}Q_{1(k)}z_{1k}z_{1k}^\top Q_{1(k)}}{1+\frac{1}{n_1}z_{1k}^\top Q_{1(k)}z_{1k}}, \label{SMW_identity1}\\
    Q_2(\lambda) = Q_{2(k)} - \frac{\frac{1}{n_2}Q_{2(k)}z_{2k}z_{2k}^\top Q_{2(k)}}{1+\frac{1}{n_2}z_{2k}^\top Q_{2(k)}z_{2k}}. \label{SMW_identity2}
\end{align}

Therefore, we can write
\begin{align}
    &\left(\frac{1}{n_1} Z_1^\top Z_1-x_1(\hat{e}_1)\right)Q_1\Theta_1Q_2\Theta_2 + \frac{\frac{n_{12}}{n_1n_2}\mathrm{Tr}(Q_1\Theta_1Q_2E_{12}^\top)}{\left(1+\frac{p_1}{n_1}\hat{e}_1\right)\left(1+\frac{p_2}{n_2}\hat{e}_2\right)}E_{12} Q_2\Theta_2\nonumber\\
    &= \Delta_1 + \Delta_2 + \Delta_3 + \Delta_4 + \Delta_5 + \Delta_6 + \Delta_7 + \Delta_8 + \Delta_9, \nonumber
\end{align}
where 
\begin{align}
    \Delta_1 \coloneqq -\frac{1}{n_1}\sum_{k=1}^{n_1}\frac{1}{1+\frac{p_1}{n_1}\hat{e}_1(-\lambda)}\left(Q_{1(k)}(\lambda)-Q_{1}(\lambda)\right)\Theta_1\left(\Tilde{Q}_{2(k)}(\lambda)-Q_2(\lambda)\right)\Theta_2,
\end{align}
\begin{align}
    \Delta_2 \coloneqq \frac{1}{n_1}\sum_{k=1}^{n_1}\frac{1}{1+\frac{p_1}{n_1}\hat{e}_1(-\lambda)}Q_{1(k)}(\lambda)\Theta_1\left(\Tilde{Q}_{2(k)}(\lambda)-Q_2(\lambda)\right)\Theta_2,
\end{align}
\begin{align}
    \Delta_3 = \frac{1}{n_1}\sum_{k=1}^{n_1}\frac{1}{1+\frac{p_1}{n_1}\hat{e}_1(-\lambda)}(Q_{1(k)}(\lambda)-Q_{1}(\lambda))\Theta_1\Tilde{Q}_{2(k)}\Theta_2, 
\end{align}
\begin{align}
    \Delta_4 = \frac{1}{n_1}\sum_{k=1}^{n_1}\frac{1}{1+\frac{p_1}{n_1}\hat{e}_1(-\lambda)}\frac{\frac{1}{n_1}\mathrm{Tr}\left\{Q_1(\lambda)\right\}-\frac{1}{n_1}z_{1k}^\top Q_{1(k)}z_{1k}}{1+\frac{1}{n}z_{1k}^\top Q_{1(k)}z_{1k}}z_{1k}z_{1k}^\top Q_{1(k)}(\lambda)\Theta_1\Tilde{Q}_{2(k)}\Theta_2,
\end{align}
\begin{align}
    \Delta_5 = \frac{1}{n_1}\sum_{k=1}^{n_1}\frac{1}{1+\frac{p_1}{n_1}\hat{e}_1(-\lambda)}\left(z_{1k}z_{1k}^\top Q_{1(k)} - Q_{1(k)}(\lambda)\right)\Theta_1\Tilde{Q}_{2(k)}\Theta_2,
\end{align}
\begin{align}
    \Delta_6 = - \frac{1}{n_{12}}\sum_{k=1}^{n_1} \frac{\frac{1}{n_2}z_{1k}^\top  Q_{1(k)}\Theta_1Q_{2(k)}z_{2k}}{\left(1+\frac{p_1}{n_1}\hat{e}_1\right)\left(1+\frac{1}{n_1}z_{1k}^\top Q_{1(k)}z_{1k}\right)\left(1+\frac{1}{n_2}z_{2k}^\top Q_{2(k)}z_{2k}\right)} \left(\frac{1}{n_1}z_{1k}^\top Q_{1(k)}z_{1k}- \frac{1}{n_1}\mathrm{Tr}(Q_1)\right)z_{1k}z_{2k}^\top Q_{2(k)} \Theta_2,
\end{align}
\begin{align}
    \Delta_7 = - \frac{1}{n_{12}}\sum_{k=1}^{n_1} \frac{\frac{1}{n_2}z_{1k}^\top  Q_{1(k)}\Theta_1Q_{2(k)}z_{2k}}{\left(1+\frac{p_1}{n_1}\hat{e}_1\right)\left(1+\frac{p_2}{n_2}\hat{e}_2\right)\left(1+\frac{1}{n_2}z_{2k}^\top Q_{2(k)}z_{2k}\right)} \left(\frac{1}{n_2}z_{2k}^\top Q_{2(k)}z_{2k}- \frac{1}{n_2}\mathrm{Tr}(Q_2)\right)z_{1k}z_{2k}^\top Q_{2(k)} \Theta_2,
\end{align}
\begin{align}
    \Delta_8 = - \frac{1}{n_{12}}\sum_{k=1}^{n_1} \frac{(\frac{1}{n_2}z_{1k}^\top  Q_{1(k)}\Theta_1Q_{2(k)}z_{2k}-\frac{1}{n_2}\mathrm{Tr}(Q_{1}\Theta_1Q_{2}E_{12}^\top ))}{\left(1+\frac{p_1}{n_1}\hat{e}_1\right)\left(1+\frac{p_2}{n_2}\hat{e}_2\right)} z_{1k}z_{2k}^\top Q_{2(k)} \Theta_2,
\end{align}
\begin{align}
    \Delta_9 = - \frac{1}{n_{12}}\sum_{k=1}^{n_1} \frac{\frac{1}{n_2}\mathrm{Tr}(Q_1\Theta_1Q_2E_{12}^\top )}{\left(1+\frac{p_1}{n_1}\hat{e}_1\right)\left(1+\frac{p_2}{n_2}\hat{e}_2\right)} \left(z_{1k}z_{2k}^\top Q_{2(k)} \Theta_2- E_{12} Q_2\Theta_2\right).
\end{align}

To prove the theorem, we demonstrate that for any $i$, $\mathrm{Tr}\{\Delta_i\} \to 0$. To this end, we bound the moments of $\Delta$'s. Let $q = 2+ \omega/6$.
\subsubsection* {Convergence of $\Delta_1$:} 
\begin{align}
    &\mathbbm{E}\left[\left|\mathrm{Tr}\left\{\frac{1}{1+\frac{p_1}{n_1}\hat{e}_1(-\lambda)}\left(Q_{1(k)}(\lambda)-Q_{1}(\lambda)\right)\Theta_1\left(\Tilde{Q}_{2(k)}(\lambda)-Q_2(\lambda)\right)\Theta_2\right\}\right|^q\right] \nonumber\\
     &\le\mathbbm{E}\left[\left|\mathrm{Tr}\left\{\frac{1}{1+\frac{p_1}{n_1}\hat{e}_1(-\lambda)}\left(Q_{1(k)}(\lambda)-Q_{1}(\lambda)\right)\Theta_1\left(Q_{2(k)}(\lambda)-Q_2(\lambda)\right)\Theta_2\right\}\right|^q\right] \nonumber\\&\le \frac{1}{n_1^{q}n_2^{q}}\mathbbm{E}\left[\left|z_{2k}^\top Q_{2(k)}\Theta_2 Q_{1(k)}z_{1k} z_{1k}^\top Q_{1(k)}\Theta_1 Q_{2(k)}z_{2k}\right|^q\right] \label{eqnlabel1}\\
    &\le \frac{1}{n_1^{q}n_2^{q}} \sqrt{\mathbbm{E}\left[\left|z_{2k}^\top Q_{2(k)}\Theta_2 Q_{1(k)}z_{1k}\right|^{2q}\right]\mathbbm{E}\left[\left|z_{1k}^\top Q_{1(k)}\Theta_1 Q_{2(k)}z_{2k} \right|^{2q}\right]} \label{eqnlabel2} \\
    & = O\left(\frac{1}{n^{q}}\right), \label{eqnlabel3}
\end{align}

where we use Lemma \ref{bai_inequality2} and the fact that $1/(1+a) \le 1$ for any $a \ge 0$ in the inequality (\ref{eqnlabel1}), we use Cauchy-Schwartz inequality in (\ref{eqnlabel2}), and we use Lemma \ref{rubio_inequality1} in the last equation (\ref{eqnlabel3}). Thus, by applying Lemma \ref{lem_for_sum_of_rv}, $\mathrm{Tr}\{\Delta_1\} \to 0$.

\subsubsection* {Convergence of $\Delta_2$:} 

\begin{align}
    &\mathbbm{E}\left[\left|\mathrm{Tr}\{\frac{1}{1+\frac{p_1}{n_1}\hat{e}_1(-\lambda)}Q_{1(k)}(\lambda)\Theta_1\left(\Tilde{Q}_{2(k)}(\lambda)-Q_2(\lambda)\right)\Theta_2\}\right|^q\right] \nonumber\\
     &\le\mathbbm{E}\left[\left|\mathrm{Tr}\{\frac{1}{1+\frac{p_1}{n_1}\hat{e}_1(-\lambda)}Q_{1(k)}(\lambda)\Theta_1\left(Q_{2(k)}(\lambda)-Q_2(\lambda)\right)\Theta_2\}\right|^q\right] \nonumber\\&\le \frac{1}{n_2^{q}}\mathbbm{E}\left[\left|z_{2k}^\top Q_{2(k)}\Theta_2 Q_{1(k)}\Theta_1 Q_{2(k)}z_{2k}\right|^q\right] \label{eqnlabel4}\\
    & = O\left(\frac{1}{n^{q}}\right), \label{eqnlabel5}
\end{align}

where we use Lemma \ref{bai_inequality2} in (\ref{eqnlabel4}), and we use Lemma \ref{rubio_inequality1} in the last equation (\ref{eqnlabel5}). Hence, from Lemma \ref{lem_for_sum_of_rv}, $\mathrm{Tr}\{\Delta_2\} \to 0$.

\subsubsection* {Convergence of $\Delta_3$:} 
$\mathrm{Tr}\{\Delta_3\} \to 0$ can be shown in a manner similar to $\Delta_2$.

\subsubsection* {Convergence of $\Delta_4$:} 
\begin{align}
    &\mathbbm{E}\left[\left|\mathrm{Tr}\left\{\frac{1}{1+\frac{p_1}{n_1}\hat{e}_1(-\lambda)}\frac{\frac{1}{n_1}\mathrm{Tr}\left\{Q_1(\lambda)\right\}-\frac{1}{n_1}z_{1k}^\top Q_{1(k)}z_{1k}}{1+\frac{1}{n_1}z_{1k}^\top Q_{1(k)}z_{1k}}z_{1k}z_{1k}^\top Q_{1(k)}(\lambda)\Theta_1\Tilde{Q}_{2(k)}\Theta_2\right\}\right|^q\right] \nonumber \\
    &\le \sqrt{\mathbbm{E}\left[\left|\frac{1}{n_1}\mathrm{Tr}\left\{Q_1(\lambda)\right\}-\frac{1}{n_1}z_{1k}^\top Q_{1(k)}z_{1k}\right|^{2q}\right] \mathbbm{E}\left[\left|z_{1k}^\top Q_{1(k)}(\lambda)\Theta_1\Tilde{Q}_{2(k)}\Theta_2z_{1k}\right|^{2q}\right]} \label{eqnlabel6}\\
    &= O\left(\sqrt{\mathbbm{E}\left[\left|\frac{1}{n_1}\mathrm{Tr}\left\{Q_1(\lambda)\right\}-\frac{1}{n_1}z_{1k}^\top Q_{1(k)}z_{1k}\right|^{2q}\right]}\right) \label{eqnlabel7}
\end{align}
where we use Cauchy-Schwartz inequality in (\ref{eqnlabel6}), and we use Lemma \ref{rubio_inequality1} in  (\ref{eqnlabel7}). From Lemma \ref{bai_ineqaulity1},
\begin{align*}
    \mathbbm{E}\left[\left|\frac{1}{n_1}\mathrm{Tr}\left\{Q_{1(k)}(\lambda)\right\}-\frac{1}{n_1}z_{1k}^\top Q_{1(k)}z_{1k}\right|^{2q}\right] \le C_q \left(\frac{1}{n_1^{q}\lambda^{2q}}\mathbbm{E}\left[|z_{11}|^{4}\right] + \frac{1}{n_1^{2q-1}\lambda^{2q}} \mathbbm{E}\left[|z_{11}|^{4q}\right]\right),
\end{align*}
where $C_q$ denotes a constant that depends only on $q$. Furthermore, based on Lemma \ref{bai_inequality2}, we obtain
\begin{align*}
    \mathbbm{E}\left[\left|\frac{1}{n_1}\mathrm{Tr}\left\{Q_{1(k)}(\lambda)\right\}-\frac{1}{n_1}\mathrm{Tr}\left\{Q_{1}(\lambda)\right\}\right|^{2q}\right] = O\left(\frac{1}{n^{2q}}\right),
\end{align*}
which can be seen in the convergence of $\Delta_2$.

Hence
\begin{align}
    \mathbbm{E}\left[\left|\mathrm{Tr}\left\{\frac{1}{1+\frac{p_1}{n_1}\hat{e}_1(-\lambda)}\frac{\frac{1}{n_1}\mathrm{Tr}\left\{Q_1(\lambda)\right\}-\frac{1}{n_1}z_{1k}^\top Q_{1(k)}z_{1k}}{1+\frac{1}{n_1}z_{1k}^\top Q_{1(k)}z_{1k}}z_{1k}z_{1k}^\top Q_{1(k)}(\lambda)\Theta_1\Tilde{Q}_{2(k)}\Theta_2\right\}\right|^q\right] =O\left(\frac{1}{n^{q/2}}\right).
\end{align}

Therefore, from Lemma \ref{lem_for_sum_of_rv}, $\mathrm{Tr}\{\Delta_4\} \to 0$.

\subsubsection* {Convergence of $\Delta_5$:}
$\mathrm{Tr}\left\{\Delta_5\right\} \to 0$ directly follows from Lemma \ref{rubio_convergence1}.

\subsubsection* {Convergence of $\Delta_6$:} 
\begin{align}
    &\mathbbm{E}\left[\left|\mathrm{Tr}\left\{\frac{\frac{1}{n_2}z_{1k}^\top  Q_{1(k)}\Theta_1Q_{2(k)}z_{2k}}{\left(1+\frac{p_1}{n_1}\hat{e}_1\right)\left(1+\frac{1}{n_1}z_{1k}^\top Q_{1(k)}z_{1k}\right)\left(1+\frac{1}{n_2}z_{2k}^\top Q_{2(k)}z_{2k}\right)} \left(\frac{1}{n_1}z_{1k}^\top Q_{1(k)}z_{1k}- \frac{1}{n_1}\mathrm{Tr}(Q_1)\right)z_{1k}z_{2k}^\top Q_{2(k)} \Theta_2\right\}\right|^q\right] \nonumber\\
    &\le  \mathbbm{E}\left[\left|\frac{1}{n_2}z_{1k}^\top  Q_{1(k)}\Theta_1Q_{2(k)}z_{2k} \left(\frac{1}{n_1}z_{1k}^\top Q_{1(k)}z_{1k}- \frac{1}{n_1}\mathrm{Tr}(Q_1)\right)z_{2k}^\top Q_{2(k)} \Theta_2 z_{1k}\right|^q\right]  \label{eqnlabel8}\\
    &\le  \mathbbm{E}^{1/3}\left[\left|\frac{1}{n_2}z_{1k}^\top  Q_{1(k)}\Theta_1Q_{2(k)}z_{2k} \right|^{3q}\right]
    \mathbbm{E}^{1/3}\left[\left| \frac{1}{n_1}z_{1k}^\top Q_{1(k)}z_{1k}- \frac{1}{n_1}\mathrm{Tr}(Q_1)\right|^{3q}\right] 
    \mathbbm{E}^{1/3}\left[\left| z_{2k}^\top Q_{2(k)} \Theta_2 z_{1k}\right|^{3q}\right] \label{eqnlabel9}\\
    &= O\left(\frac{1}{n^{q/2}}\right) \label{eqnlabel10}
\end{align}
where we use the fact $1/(1+a) \le 1$ for any $a \ge 0$ in (\ref{eqnlabel8}), we use generalized H\"older's inequality in (\ref{eqnlabel9}), and we use Lemma \ref{rubio_inequality1}, Lemma \ref{bai_ineqaulity1} and Lemma \ref{bai_inequality2} in (\ref{eqnlabel10}). 
Therefore, from Lemma \ref{lem_for_sum_of_rv}, $\mathrm{Tr}\{\Delta_6\} \to 0$.

\subsubsection* {Convergence of $\Delta_7$:} 
$\mathrm{Tr}\{\Delta_7\} \to 0$ can be demonstrated, in a manner analoguous to $\Delta_6$.

\subsubsection* {Convergence of $\Delta_8$:} 

\begin{align}
    &\mathbbm{E}\left[\left|\mathrm{Tr}\left\{\frac{(\frac{1}{n_2}z_{1k}^\top  Q_{1(k)}\Theta_1Q_{2(k)}z_{2k}-\frac{1}{n_2}\mathrm{Tr}(Q_{1}\Theta_1Q_{2}E_{12}^\top ))}{\left(1+\frac{p_1}{n_1}\hat{e}_1\right)\left(1+\frac{p_2}{n_2}\hat{e}_2\right)} z_{1k}z_{2k}^\top Q_{2(k)} \Theta_2\right\}\right|^q\right] \nonumber\\
    &\le \sqrt{\mathbbm{E}\left[\left|\left(\frac{1}{n_2}z_{1k}^\top  Q_{1(k)}\Theta_1Q_{2(k)}z_{2k}-\frac{1}{n_2}\mathrm{Tr}(Q_{1}\Theta_1Q_{2}E_{12}^\top )\right)\right|^{2q}\right]
    \mathbbm{E}\left[\left|z_{2k}^\top Q_{2(k)} \Theta_2z_{1k}\right|^{2q}\right]} \label{eqnlabel11},
\end{align}
where we use Cauchy-Schwartz inequality in (\ref{eqnlabel11}). As shown in the part of $\Delta_1, \Delta_2$, 
\begin{align*}
    \mathbbm{E}\left[\left|\left(\frac{1}{n_2}\mathrm{Tr}(Q_{1(k)}\Theta_1 (Q_{2}- Q_{2(k)})E_{12}^\top )\right)\right|^{2q}\right] = O(\frac{1}{n^{2q}}),
\end{align*}
\begin{align*}
    \mathbbm{E}\left[\left|\left(\frac{1}{n_2}\mathrm{Tr}(Q_{1}- Q_{1(k)})\Theta_1 Q_{2(k)}E_{12}^\top )\right)\right|^{2q}\right] = O(\frac{1}{n^{2q}}),
\end{align*}
\begin{align*}
    \mathbbm{E}\left[\left|\left(\frac{1}{n_2}\mathrm{Tr}(Q_{1}- Q_{1(k)})\Theta_1 (Q_{2}- Q_{2(k)})E_{12}^\top )\right)\right|^{2q}\right] = O(\frac{1}{n^{4q}}),
\end{align*}
and, from Lemma \ref{bai_ineqaulity1},
\begin{align*}
    \mathbbm{E}\left[\left|\left(\frac{1}{n_2}z_{1k}^\top  Q_{1(k)}\Theta_1Q_{2(k)}z_{2k}-\frac{1}{n_2}\mathrm{Tr}(Q_{1(k)}\Theta_1Q_{2(k)}E_{12}^\top )\right)\right|^{2q}\right] = O\left(\frac{1}{n^{q}}\right),
\end{align*}
we obtain
\begin{align}
    \mathbbm{E}\left[\left|\mathrm{Tr}\left\{\frac{(\frac{1}{n_2}z_{1k}^\top  Q_{1(k)}\Theta_1Q_{2(k)}z_{2k}-\frac{1}{n_2}\mathrm{Tr}(Q_{1}\Theta_1Q_{2}E_{12}^\top ))}{\left(1+\frac{p_1}{n_1}\hat{e}_1\right)\left(1+\frac{p_2}{n_2}\hat{e}_2\right)} z_{1k}z_{2k}^\top Q_{2(k)} \Theta_2\right\}\right|^q\right] = O\left(\frac{1}{n^{q/2}}\right)
\end{align}
Therefore, $\mathrm{Tr}\{\Delta_8\} \to 0$ follows from Lemma \ref{lem_for_sum_of_rv}.

\subsubsection* {Convergence of $\Delta_9$:}
We decompose $\Delta_9$ as $\Delta_9 = \Delta_{91} + \Delta_{92}$, where
\begin{align}
    \Delta_{91} = - \frac{1}{n_1}\sum_{k=1}^{n_1} \frac{\frac{1}{n_2}\mathrm{Tr}(Q_1\Theta_1Q_2E_{12}^\top )}{\left(1+\frac{p_1}{n_1}\hat{e}_1\right)\left(1+\frac{p_2}{n_2}\hat{e}_2\right)} \left(z_{1k}z_{2k}^\top Q_{2(k)} \Theta_2 - E_{12} Q_{2(k)}\Theta_2\right)
\end{align}
and
\begin{align}
    \Delta_{92} = - \frac{1}{n_1}\sum_{k=1}^{n_1} \frac{\frac{1}{n_2}\mathrm{Tr}(Q_1\Theta_1Q_2E_{12}^\top )}{\left(1+\frac{p_1}{n_1}\hat{e}_1\right)\left(1+\frac{p_2}{n_2}\hat{e}_2\right)} \left(E_{12} Q_{2(k)}\Theta_2 - E_{12} Q_{2}\Theta_2\right).
\end{align}

$\mathrm{Tr}\{\Delta_{91}\} \to 0$ follows from Lemma \ref{rubio_convergence1}. $\mathrm{Tr}\{\Delta_{92}\} \to 0$ can be shown in a manner similar to $\Delta_2$.\\ \\

Since
\begin{align}
    &\frac{1}{x_1(\hat{e}_1)+\lambda}\left(\frac{1}{n_1} Z_1^\top Z_1-x_1(\hat{e}_1)\right)Q_1\Theta_1Q_2\Theta_2 + \frac{1}{x_1(\hat{e}_1)+\lambda}\frac{\frac{n_{12}}{n_1n_2}\mathrm{Tr}(Q_1\Theta_1Q_2E_{12}^\top)}{\left(1+\frac{p_1}{n_1}\hat{e}_1\right)\left(1+\frac{p_2}{n_2}\hat{e}_2\right)}E_{12} Q_2\Theta_2\nonumber\\
    &=\frac{1}{x_1(\hat{e}_1)+\lambda}\Theta_1Q_2\Theta_2 +  \frac{1}{x_1(\hat{e}_1)+\lambda}\frac{\frac{n_{12}}{n_1n_2}\mathrm{Tr}(Q_1\Theta_1Q_2E_{12}^\top)}{\left(1+\frac{p_1}{n_1}\hat{e}_1\right)\left(1+\frac{p_2}{n_2}\hat{e}_2\right)}E_{12} Q_2\Theta_2 -  Q_1\Theta_1Q_2\Theta_2, \nonumber
\end{align}
we can combine these results to obtain
\begin{align}
\label{eqn:first_step}
    \frac{1}{x_1(\hat{e}_1)+\lambda}\mathrm{Tr}\left\{\Theta_1Q_2\Theta_2\right\} + \frac{1}{x_1(\hat{e}_1)+\lambda}\frac{\mathrm{Tr}\left\{E_{12} Q_2\Theta_2\right\}\frac{n_{12}}{n_1n_2}\mathrm{Tr}(Q_1\Theta_1Q_2E_{12}^\top)}{\left(1+\frac{p_1}{n_1}\hat{e}_1\right)\left(1+\frac{p_2}{n_2}\hat{e}_2\right)} - \mathrm{Tr}\left\{Q_1\Theta_1Q_2\Theta_2\right\}\to 0.
\end{align}

From Theorem 1 of \cite{rubio2011spectral},

\begin{align}
\label{eqnlabel2.1}
    \frac{1}{x_1(\hat{e}_1)+\lambda}\mathrm{Tr}\left\{\Theta_1Q_2\Theta_2\right\} - \frac{1}{x_1(\hat{e}_1)+\lambda}\frac{1}{x_2(e_2)+\lambda}\mathrm{Tr}\left\{\Theta_1\Theta_2\right\} \to 0,
\end{align}
\begin{align}
\label{eqnlabel2.2}
    \mathrm{Tr}\left\{E_{12}Q_2\Theta_2\right\} - \frac{1}{x_2(e_2)+\lambda}\mathrm{Tr}\left\{E_{12}\Theta_2\right\} \to 0.
\end{align}

By combining (\ref{eqnlabel2.1}) and (\ref{eqnlabel2.2}) with (\ref{eqn:first_step}), we obtain
\begin{align}
\label{eqn:first_step2}
    \frac{1}{x_1(\hat{e}_1)+\lambda}\frac{1}{x_2(e_2)+\lambda}\mathrm{Tr}\left\{\Theta_1\Theta_2\right\} + \frac{1}{x_1(\hat{e}_1)+\lambda}\frac{1}{x_2(e_2)+\lambda}\frac{\mathrm{Tr}\left\{E_{12}\Theta_2\right\}\frac{n_{12}}{n_1n_2}\mathrm{Tr}(Q_1\Theta_1Q_2E_{12}^\top)}{\left(1+\frac{p_1}{n_1}\hat{e}_1\right)\left(1+\frac{p_2}{n_2}\hat{e}_2\right)} - \mathrm{Tr}\left\{Q_1\Theta_1Q_2\Theta_2\right\}\to 0.
\end{align}

It is well-established that $\hat{e}_1 - e_1 \to 0$, $\hat{e}_2 - e_2\to 0$, $e_1(-\lambda) = \frac{1}{x(e_1)+\lambda}$ and $e_2(-\lambda) = \frac{1}{x(e_2)+\lambda}$; e.g., see Theorem 1 of \cite{rubio2011spectral}. Therefore, 
comparing (\ref{eqn:first_step2}) with our claim of this Lemma, we only have to prove that
\begin{align}
\label{sublemma}
    \frac{n_{12}}{n_1n_2}\mathrm{Tr}(Q_1\Theta_1Q_2E_{12}^\top) - m_{12}^*(\lambda) \to 0. 
\end{align}
This is easily observed in (\ref{eqn:first_step2}). By substituting $n_{12}E_{12}^\top/n_1n_2$ into $\Theta_2$, we obtain 
\begin{align}
\label{eqn:second_step}
    \frac{n_{12}}{n_1n_2}\frac{1}{x_1(\hat{e}_1)+\lambda}\frac{1}{x_2(\hat{e}_2)+\lambda}\mathrm{Tr}\left\{\Theta_1E_{12}^\top\right\} + \frac{n_{12}p_{12}}{n_1n_2}\frac{1}{x_1(\hat{e}_1)+\lambda}\frac{1}{x_2(\hat{e}_2)+\lambda}\frac{\frac{n_{12}}{n_1n_2}\mathrm{Tr}(Q_1\Theta_1Q_2E_{12}^\top)}{\left(1+\frac{p_1}{n_1}\hat{e}_1\right)\left(1+\frac{p_2}{n_2}\hat{e}_2\right)} - \frac{n_{12}}{n_1n_2}\mathrm{Tr}\left\{Q_1\Theta_1Q_2E_{12}^\top\right\}\to 0.
\end{align}
Therefore, (\ref{sublemma}) holds. Finally, we need to prove the existence of $m_{12}^*$. To this end, we have to prove 
\begin{align}
    \frac{\frac{n_{12}p_{12}}{n_1n_2}e_1(-\lambda)e_2(-\lambda)}{\left(1+\frac{p_1}{n_1}e_1(-\lambda)\right)\left(1+\frac{p_2}{n_2}e_2(-\lambda)\right)} < 1
\end{align}
for all $\lambda > 0$. Since the mapping $x \mapsto x/(1+x)$ is monotonically increasing function, and each $e_l(-\lambda)$ is monotonically decreasing function for $l=1,2$, we only need to ensure that 
\begin{align}
\label{sub2}
    1 > \lim_{\lambda\to0} \frac{\frac{n_{12}p_{12}}{n_1n_2}e_1(-\lambda)e_2(-\lambda)}{\left(1+\frac{p_1}{n_1}e_1(-\lambda)\right)\left(1+\frac{p_2}{n_2}e_2(-\lambda)\right)}.
\end{align}
Since 
\begin{align}
    \lim_{\lambda\to0} \frac{\frac{n_{12}p_{12}}{n_1n_2}e_1(-\lambda)e_2(-\lambda)}{\left(1+\frac{p_1}{n_1}e_1(-\lambda)\right)\left(1+\frac{p_2}{n_2}e_2(-\lambda)\right)} =
    \begin{cases}
        \frac{n_{12}p_{12}}{n_1n_2}  &\quad p_1 < n_1, p_2 < n_2,\\
        \frac{n_{12}p_{12}}{n_2 p_1} &\quad n_1 < p_1, p_2 < n_2,\\
        \frac{n_{12}p_{12}}{n_1p_2} &\quad p_1 < n_1, n_2 < p_2,\\
        \frac{n_{12}p_{12}}{p_1 p_2} &\quad n_1 < p_1, n_2 < p_2,
    \end{cases}
\end{align}
(\ref{sub2}) holds true.
By combining all above results, the proof is complete.

\end{proof}

\begin{remark}
    The above Lemma \ref{important_lemma} assumes that the common parts of $Z1$ and $Z2$ are clustered in the upper left corner of the matrix, but due to symmetry, the above Lemma \ref{important_lemma} holds for both matrices whose common parts are completely scattered by rearranging the rows and columns of both matrices, respectively.
\end{remark}

The following Lemma is also helpful in obtaining the limiting behavior of the bias and variance terms of the risks in Theorem \ref{main_result1} and \ref{main_result2}. 

\begin{lemma}
\label{fundamental lemma}
Let $Q_1 \coloneqq Q_1(\lambda)\coloneqq \left(n_{1}^{-1}Z_1^\top Z_1 +\lambda \mathrm{I}_{p_1}\right)^{-1}$ and $Q_2\coloneqq Q_2(\lambda)\coloneqq \left(n_2^{-1}Z_2^\top Z_2 +\lambda \mathrm{I}_{p_2}\right)^{-1} $, where $Z_1$ and $Z_2$ be $n_1\times p_1$ and $n_2\times p_2$ matrices satisfying all assumptions made in Lemma \ref{important_lemma}. Furthermore, let $Z_{1,12}$ and $Z_{2,12}$ be $n_{12}\times p_1$ and $n_{12}\times p_2$ matrices, respectively, which are taken from the parts of the $Z_1$ and $Z_2$ matrices corresponding to the common $n_12$ row vectors. (Under the assumption made in Lemma \ref{important_lemma}, $Z_{1,12}$ and $Z_{2,12}$ are taken from the first $n_{12}$ row vectors of $Z_1$ and $Z_2$, respectively). Then, for any $\lambda>0$,
\begin{align*}
    \frac{1}{n_1n_2} \mathrm{Tr} \left(Z_{1,12}Q_1E_{12}Q_2 Z_{2,12}^\top\right) - \frac{\frac{n_{12}}{n_1n_2}\mathrm{Tr} \left(Q_1E_{12} Q_2 E_{12}^\top\right)}{\left(1+\frac{1}{n_1}\mathrm{Tr}\left(Q_1\right)\right)\left(1+\frac{1}{n_2}\mathrm{Tr}\left(Q_2\right)\right)}\to 0
\end{align*}
almost surely, where $E_{12} = \frac{1}{n}\mathbbm{E}\left[Z_1^\top Z_2\right]$.
\end{lemma}

\begin{proof}
    In this proof, we use the same notation as that in the proof of Lemma \ref{important_lemma}. 
    First, note that
    \begin{align}
        \frac{1}{n_1n_2} \mathrm{Tr} \left(Z_{1,12}Q_1E_{12}Q_2 Z_{2,12}^\top\right) &= \frac{1}{n_1n_2}\sum_{k=1}^{n_{12}} z_{1,12,k}^\top Q_1 E_{12} Q_2 z_{2,12,k} \nonumber   \\
        &= \frac{n_{12}}{n_1n_2}\sum_{k=1}^{n_{12}} \frac{\frac{1}{n_{12}}z_{1,k}^\top Q_{1(k)} E_{12} Q_{2(k)} z_{2k}}{(1+\frac{1}{n_1}z_{1k}^\top Q_{1(k)} z_{1k}) (1+\frac{1}{n_2}z_{2k}^\top Q_{2(k)} z_{2k})} ,\label{eqnlabel3.1}
    \end{align}
    where we use the well-known formula,
    \begin{align}
        Q_{1}z_{1k} =  \frac{Q_{1(k)}z_{1k}}{1+\frac{1}{n_1}z_{1k}^\top Q_{1(k)} z_{1k}}
    \end{align}
    \begin{align}
        Q_{2}z_{2k} =  \frac{Q_{2(k)}z_{2k}}{1+\frac{1}{n_2}z_{2k}^\top Q_{2(k)} z_{2k}}
    \end{align}
    which can be deduced from the Sherman-Morrison-Woodbury identity.
    Hence, using an argument analogous to the proof of Lemma \ref{important_lemma}, 
    \begin{align*}
        \frac{1}{n_1n_2} \mathrm{Tr} \left(Z_1Q_1E_{12}Q_2 Z_2^\top\right) - \frac{\frac{n_{12}}{n_{1}n_2}\mathrm{Tr} \left(Q_1E_{12} Q_2 E_{12}^\top\right)}{\left(1+\frac{1}{n_1}\mathrm{Tr}\left(Q_1\right)\right)\left(1+\frac{1}{n_2}\mathrm{Tr}\left(Q_2\right)\right)}\to 0.
    \end{align*}
    Indeed, from Lemmas \ref{bai_ineqaulity1},\ref{bai_inequality2}, and \ref{rubio_inequality1}, we can obtain
    \begin{align*}
        \mathbbm{E}\left[\left|\frac{\frac{1}{n_{12}}z_{1k}^\top Q_{1(k)} E_{12} Q_{2(k)} z_{2k}}{(1+\frac{1}{n_1}z_{1k}^\top Q_{1(k)} z_{1k}) (1+\frac{1}{n_2}z_{2k}^\top Q_{2(k)} z_{2k})}- \frac{\frac{1}{n_{12}}\mathrm{Tr} \left(Q_1E_{12} Q_2 E_{12}^\top\right)}{(1+\frac{1}{n_1}z_{1k}^\top Q_{1(k)} z_{1k}) (1+\frac{1}{n_2}z_{2k}^\top Q_{2(k)} z_{2k})}\right|^q\right] = O\left(\frac{1}{n^{q/2}}\right),
    \end{align*}
    
    \begin{align*}
        \mathbbm{E}\left[\left|\frac{\frac{1}{n_{12}}\mathrm{Tr} \left(Q_1E_{12} Q_2 E_{12}^\top\right)}{(1+\frac{1}{n_1}z_{1k}^\top Q_{1(k)} z_{1k}) (1+\frac{1}{n_{2}}z_{2k}^\top Q_{2(k)} z_{2k})}- \frac{\frac{1}{n_{12}}\mathrm{Tr} \left(Q_1E_{12} Q_2 E_{12}^\top\right)}{\left(1+\frac{1}{n_1}\mathrm{Tr}\left(Q_1\right)\right) (1+\frac{1}{n_2}z_{2k}^\top Q_{2(k)} z_{2k})}\right|^q\right] = O\left(\frac{1}{n^{q/2}}\right),
    \end{align*}
    
    \begin{align*}
        \mathbbm{E}\left[\left|\frac{\frac{1}{n_{12}}\mathrm{Tr} \left(Q_1E_{12} Q_2 E_{12}^\top\right)}{\left(1+\frac{1}{n_1}\mathrm{Tr}\left(Q_1\right)\right) (1+\frac{1}{n_2}z_{2k}^\top Q_{2(k)} z_{2k})}- \frac{\frac{1}{n_{12}}\mathrm{Tr} \left(Q_1E_{12} Q_2 E_{12}^\top\right)}{\left(1+\frac{1}{n_1}\mathrm{Tr}\left(Q_1\right)\right)\left(1+\frac{1}{n_2}\mathrm{Tr}\left(Q_2\right)\right)}\right|^q\right] = O\left(\frac{1}{n^{q/2}}\right),
    \end{align*}
    with $q = 2+\omega/6$.
    This completes the proof together with Lemma \ref{lem_for_sum_of_rv}.
    
\end{proof}
 
\end{document}